\documentclass[12pt]{article}
\usepackage{amssymb}
\usepackage{amsthm}
\usepackage{amsmath}  

\usepackage[colorlinks,citecolor=blue,urlcolor=blue,breaklinks=true]{hyperref}
\usepackage[noabbrev,nameinlink,capitalize]{cleveref}

\usepackage{authblk}

\newtheorem{theorem}{Theorem}[section]
\newtheorem{proposition}[theorem]{Proposition}
\newtheorem{lemma}[theorem]{Lemma}
\newtheorem{corollary}[theorem]{Corollary}

\newtheorem{definition}[theorem]{Definition}
\newtheorem{example}[theorem]{Example}

\theoremstyle{remark}

\newtheorem{remark}[theorem]{Remark}

\theoremstyle{plain}
\newtheorem{assumption}[theorem]{Assumption}

\numberwithin{equation}{section}

\usepackage{xargs}
\usepackage[pdftex,dvipsnames]{xcolor}

\usepackage{marginnote}

\usepackage[mode=buildnew]{standalone}
\usepackage{tikz,tikz-cd}
\usetikzlibrary{decorations.pathmorphing} 
\usepackage{graphicx}
\graphicspath{ {./figures/} }

\usepackage{framed}

\usepackage{caption}
\usepackage{subcaption}

\usepackage{enumitem}
\usepackage{cancel}
\usepackage{pifont}

\DeclareMathOperator*{\argmin}{arg\,min}

\newcommand{\dd}{\ensuremath{\,\mathrm d}}

\DeclareMathOperator{\E}{\mathbb E}
\DeclareMathOperator{\Lip}{Lip}

\newcommand{\N}{\mathbb{N}}

\newcommand{\R}{\mathbb{R}}

\DeclareMathOperator{\Size}{size}
\DeclareMathOperator{\Depth}{depth}
\DeclareMathOperator{\ReLU}{ReLU}
\DeclareMathOperator{\supp}{supp}

\definecolor{CharlesColor}{HTML}{990011}

\usepackage{arxiv}

\title{Functional SDE approximation inspired by a deep operator network architecture}
\author[1]{Martin Eigel}
\author[2]{Charles Miranda}
\affil[1]{Weierstrass Institute for Applied Analysis  and Stochastics, Berlin, Germany \thanks{\texttt{eigel@wias-berlin.de, charles.miranda@ec-nantes.fr}}}
\affil[2]{Centrale Nantes, Nantes Université, Laboratoire de Mathématiques Jean Leray UMR CNRS 6629, France}
\renewcommand{\shorttitle}{SDEONet}
\date{\today}
\begin{document}
    
    \maketitle
    \begin{abstract}
       A novel approach to approximate solutions of Stochastic Differential Equations (SDEs) by Deep Neural Networks (DNN) is presented.
The architecture is inspired by the notion of Deep Operator Networks (DeepONets), which are based on operator learning in function spaces in terms of a reduced basis also represented in the network. In our setting, we make use of a polynomial chaos expansion (PCE) of stochastic processes and call the corresponding architecture SDEONet.
This construction can be conceived as a theoretical tool, which enables the derivation of new DNN convergence and complexity results for the solution of SDEs with possible extensions to numerical solutions of nonlinear SDEs and Stochastic Partial Differential Equations (SPDEs).
       
The PCE has been used extensively in the area of uncertainty quantification (UQ) for parametric partial differential equations. This however is not the case with SDEs, where classical sampling methods dominate and functional approaches are seen rarely.
A main challenge with truncated PCEs occurs due to the drastic growth of the number of components with respect to the maximum polynomial degree and the number of basis elements.
The proposed SDEONet architecture aims to alleviate the issue of exponential complexity by learning an optimal sparse truncation of the Wiener chaos expansion. A complete convergence and complexity analysis is presented, making use of recent Neural Network approximation results.
While intended as a theoretical approach to foster a DNN analysis, some numerical experiments illustrate the practical performance of the suggested approach in 1D and higher dimensions.
    \end{abstract}

    \section{Introduction}
\label{sec:intro}

Stochastic differential equations (SDEs)~\cite{ksendal2003} extend ordinary differential equations (ODEs) by incorporating stochastic terms, resulting in stochastic process trajectories and necessitating integration with respect to white noise as formalised by It\=o. They are extensively applied in many technical fields such as molecular dynamics, financial mathematics, and more recently, deep learning, particularly in generative modelling~\cite{yang2023diffusion}.
A common numerical approach for solving SDEs is the Euler-Maruyama scheme, which is an adaptation of the explicit Euler scheme for ODEs. While its implementation and analysis are relatively straightforward, certain conditions are required from the SDE, which restrict the applicability of the scheme. Deep neural network (DNN) architectures can be designed to represent this method directly, as e.g. demonstrated in~\cite{E2017,miranda2024approximatinglangevinmontecarlo}.
In this work, however, we explore an alternative functional representation of SDEs utilising a polynomial chaos expansion (PCE). This employs tensorised orthogonal polynomials with respect to a probability measure as a basis.
Known convergence results of the PCE lead to a novel approach to a DNN convergence analysis for representing SDE solutions.
While we do not fully exploit it in this work, the proposed approach is more general than the EM scheme and can be used with a broader range of stochastic equations.
The Cameron-Martin theorem~\cite{cameron1947orthogonal,janson1997gaussian} states that any stochastic process with finite variance can be represented by a PCE. This representation has been widely employed in the analysis of parametric PDEs and their numerical treatment in the area of Uncertainty Quantification (UQ)~\cite{schwab2011sparse,cohen2015approximation,ernst2012convergence}.
Although functional approximations offer practical advantages for SDEs such as direct access to (approximations of) statistical quantities, stochastic pathwise Monte Carlo sampling methods are clearly predominant~\cite{holden1996stochastic}.
A principal reason for this is that functional approximations of irregular processes typically necessitate a ``fine'' discretisation for accuracy, involving many expansion terms and high polynomial degrees. This complexity, further compounded by the ``curse of dimensionality'', necessitates modern compression tools and complicates numerical scheme analysis.

In this paper, we tackle the functional approach utilising a DNN architecture called Deep Operator Network (DeepONet)~\cite{luDeepONetLearningNonlinear2021}, designed to learn operators in infinite-dimensional function spaces. DeepONets are composed of trunk and branch networks, representing a learning reduced basis and respective representation coefficients.
Mathematically, this architecture permits a comprehensive theoretical analysis as presented in~\cite{lanthalerErrorEstimatesDeepONets2022}. While DeepONets and similar methods like Fourier Neural Operators have been applied to (partial) differential equations~\cite{goswami2023physics,kovachki2023neural}, their application to SDEs is a new application.
By employing the reduced basis exploration of DeepONet with the PCE of an SDE as analysed in~\cite{huschtoAsymptoticErrorChaos2019}, we achieve a compressed functional SDE representation where the deterministic coefficients follow an appropriate ODE.
The structure of the architecture is illustrated in~\cref{fig:sketch_sdeonet}, where a Brownian motion $W$ is encoded by $\mathcal E$, resulting in the input $G$ of approximation $\mathcal A$ and reconstruction $\mathcal R$. The network's output is the SDE solution $X_t$. 

\begin{figure}
    \centering
    \includegraphics[scale=0.85]{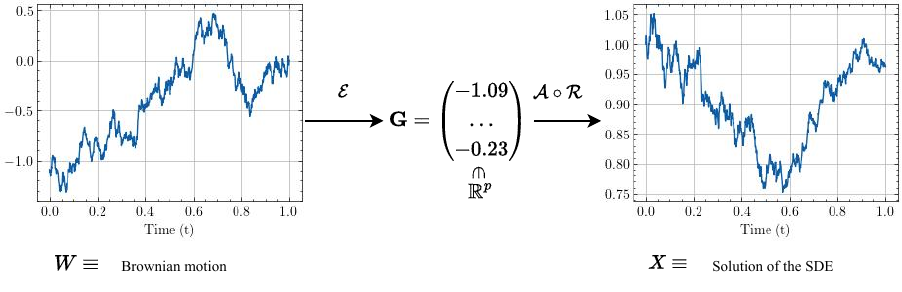}    
    \caption{Sketch of the SDEONet action working on input $G$ (processed Brownian motion, left), approximation $\mathcal A$ and reconstruction $\mathcal R$ of the SDE solution (right).}
    \label{fig:sketch_sdeonet}
\end{figure}

To describe our setting, assume the filtered probability space $(\Omega, \mathcal F, (\mathcal F_t)_{t \in [0,T]}, \mathbb P)$ with Brownian motion $(W_t)_{t \in [0,T]}$.
As a model problem, we consider the continuous stochastic process $(X_t)_{t \in [0,T]}$ that satisfies the SDE given by
\begin{equation}
    \dd X_t = \mu(t, X_t)\dd t + \sigma(t, X_t)\dd W_t, \quad\text{with}~X_0=x_0.
    \label{eq:sde}
\end{equation}

Our main contributions are:
\begin{itemize}
    \item Development of a DeepONet inspired architecture for the functional (Wiener chaos) representation of SDE solutions.
    \item Convergence and complexity analysis of this architecture in terms of the discretization parameters.
    To achieve this, recent results on NN approximations of polynomials and H\"older continuous functions are used~\cite{Schwab2023,Opschoor2021,Petersen2018}.
\end{itemize}

We present our main result~\cref{theorem:main_result} qualitatively for the devised SDEONet architecture in the following.

\begin{theorem}[Neural network approximation of a strong solution of a SDE]\label{theorem:main_result}
    Let $p,m=2^k \in \N$, $\mathcal G$ be a SDE solution operator (\cref{def:sde_sol_operator}) and $\varepsilon \in \left(0, \left[\frac{4}{e^2 T}\right]^{1/3}\right)$.
    Then, there exists a SDEONet $\mathcal N^{p,m}$ (\cref{def:sde_deeponet}) that satisfies
    \begin{align*}
        \hat E \leq & \min_{(q,\ell) \in J_p}\left( (1+x_0^2)\left(\int_0^T C_3(t, K)\left(\frac{1}{(q+1)!} + \frac{2T(1+t)}{\ell}\right) \dd t\right)\right)^{1/2}\\
        & + \sqrt{\varepsilon} \min_{(q,\ell) \in J_p} \left((1+x_0^2)\left(C_4(K,T) - \frac{1}{(q+1)!}\frac{T(C_5 T)^{2(q+1)}}{2(q+1)+1}\left(1+\frac{1}{C_5 T}\right)^{q+1}\right)\right)^{1/2}\\
        & + \varepsilon\sqrt{2(\varepsilon + p)},
    \end{align*}
    with $C_3(t,K)$ defined in \cref{theorem:l2_error_truncation}, $C_4(K,T):=\int_0^T e^{Bte^{Bt}} \dd t$ and $C_5=C_5(K,T)$ are the constants in \cref{lemma:ub_approx_error}.

    The SDEONet is composed of an \emph{approximator} $\mathcal A$ that satisfies
    \begin{align*}
        \Size(\mathcal A) &\leq C_1 p \max_{j \in \{1,\dots,p\}}|k_j^*|^3 \log(1+|k_j^*|)|k_j^*|_0^2 \log(p\varepsilon^{-1}),\\
        \Depth(\mathcal A) &\leq C_1 \max_{j \in \{1,\dots,p\}} |k_j^*|\log(1+|k_j^*|)^2|k_j^*|_0\log(1+|k_j^*|_0)\log(p\varepsilon^{-1}),
    \end{align*}
    with a $C_1 > 0$ independent of $p, m$, $\varepsilon$ and of a \emph{trunk net} $\tau$ that satisfies
    \begin{align*}
        \Depth(\tau) \leq (2+\lceil \log_2(n+1)\rceil)(12+n),\quad
        \Size(\tau) \leq C_2 p\left(\frac{\varepsilon}{\sqrt T}\right)^{-\frac{1}{n+1}},
    \end{align*}
    with a $C_2 > 0$ that depends only on the regularity of the coefficients $x_j$ of the polynomial chaos expansion \cref{eq:sde_params}.
\end{theorem}

The central objective of this work is to present a novel approach to the representation of functional solutions for SDEs with NNs that can be mathematically analysed. Previous attempts to represent and analyse SDE solutions with NNs include learning the Kolmogorov backward SDE~\cite{beck2021solving} on the basis of an Euler discretisation scheme or on a Picard iteration.
Additionally, a ResNet-based architecture for Langevin SDE~\cite{eigel2023approximating} was presented, which was utilised for the purpose of interacting particle transport in the context of Bayesian inverse problems~\cite{garbuno2020affine,eigel2022less}.

An alternative compression technique to DNNs can be found in low-rank tensor formats~\cite{hackbusch2014numerical,nouy2017low}, in particular with modern hierarchical formats such as tensor trains (TT)~\cite{oseledets2011tensor}.
In~\cite{bayer2023pricing}, it was demonstrated that a TT based Longstaff-Schwarz algorithm is comparable to equivalent NN constructions~\cite{becker2019deep}.

Despite its unfavourable complexity for simple linear SDEs, the proposed scheme can be used for more general problems than the Euler-Maruyama scheme. In order to utilise the latter, it is necessary to have explicit access to the drift and diffusion terms.
For example in case of SPDEs, these terms are directly dependent on the differentials of the solution and such a numerical scheme hence would not be applicable.

The paper structure includes preliminary reviews of SDEs and their approximation in terms of Wiener chaos expansions in~\Cref{sec:sde}, an introduction to NN and DON in~\Cref{sec:DON}, and a description of the new SDEONet architecture in~\Cref{sec:sdeonet}.
The convergence analysis in~\Cref{sec:analysis} of this architecture combines results from Wiener chaos and NN approximation of polynomials and functions with H\"older regularity. The practical performance of SDEONets is illustrated experimentally in~\Cref{sec:experiments} based on benchmark problems in one and more dimensions.
    \section{SDEs and Wiener chaos}
\label{sec:sde}

In this section, we review polynomial chaos representations of stochastic processes in terms of Hermite polynomials, as presented in~\cite{Huschto2014}.
We introduce the specific setting we use for the construction of the SDEONet architecture in~\Cref{sec:DON}.
First, we recall the SDE setting and requirements for well-posedness and regularity of the solution.
 
\subsection{Definitions and notation}
Given a probability space $(\Omega, \mathcal F, \mathbb P)$ and a $\R^d$-valued Brownian motion $W$, we consider
\begin{itemize}
    \item $\{(\mathcal F_t); t \in [0,T]\}$, the filtration generated by the Brownian motion $W$,
    \item $L^p(\mathcal F) := L^p(\Omega, \mathcal F, \mathbb P), p \in \N^*$, the space of all $\mathcal F$-measurable random variables (r.v.) $X : \Omega \to \R^d$ satisfying $\|X\|_p^p := \E[\|X\|_{\ell_p}^p] < \infty$,
    \item $C^{p,q}(U \times V, W), p,q \in \N \cup \{\infty\}$, the space of functions $f : U \times V \to W$ that are $p$ continuously differentiable in the first component and $q$ continuously differentiable in the second component,
    \item $C^\beta(K) := \{f \in C^n(K) : \|f\|_{C^\beta} \leq \infty\}$, with $K \subset \R^d$ is compact and $\beta = (n,\xi) \in \N \times (0,1]$, the space of $\beta$-H\"older continuous functions, where
    \begin{align*}
        \|f\|_{C^\beta} &:= \max\left\{\max_{|\alpha| \leq n}\|\partial^\alpha f\|_\infty, \max_{|\alpha| = n}\Lip_\xi(\partial^\alpha f) \right\}\\
        \Lip_\xi(f) &:= \sup_{x \neq y \in K}\frac{|f(x)-f(y)|}{|x-y|^\xi}.
    \end{align*}
\end{itemize}

To ensure the uniqueness of the solution of \cref{eq:sde}, we henceforth require the following assumptions to be satisfied.
\begin{assumption}\label{assump:uniqueness_sde}
    Let $T>0$ and $\mu:[0,T]\times \R^d \to \R^d$, $\sigma : [0,T] \times \R^d \to \R^{d\times m}$ be $\mathcal F$-measurable functions satisfying,
    \begin{enumerate}
        \item linear growth: $\forall x \in \R^d, t \in [0,T], \|\mu(t,x)\|_2 + \|\sigma(t,x)\|_{\text{Fro}} \leq K(1+\|x\|_2)$,
        \item uniform Lipschitz continuity: $\forall x,y \in \R^d, t \in [0,T], \|\mu(t,x)-\mu(t,y)\|_2 + \|\sigma(t,x)-\sigma(t,y)\|_\emph{Fro} \leq K\|x-y\|_2$
    \end{enumerate}
\end{assumption}

Under these assumptions, \cite[Theorem~5.2.1]{ksendal2003} asserts that the SDE \cref{eq:sde} has a unique $t$-continuous solution $X_t$.
As a preparation for the analysis, we now introduce a notion of multiple stochastic integrals.
For $f \in L^2([0,T]^n)$ a symmetric function, the $n$-multiple stochastic integral $I_n(f)$ is defined as the stochastic integral,
\begin{equation}
    I_n(f):= \int_0^T \int_0^{t_n} \dots \int_0^{t_2} f(t_1,\dots,t_n) \dd W_{t_1}\dots \dd W_{t_n}.
    \label{eq:multiple_stochastic_integrals}
\end{equation}

\subsection{Wiener chaos expansion of SDEs}
\label{sec:wiener chaos}

In this section we recall the notion of the Wiener chaos expansion and convergence results required later.
For more information, we refer to~\cite{MR2200233,janson1997gaussian}.

Normalised Hermite polynomials are defined through the identities
\begin{align}
    H_0(x) := 1\text{ and }
    H_n(x) := \frac{(-1)^n}{\sqrt{n!}} \exp\left(\frac{x^2}{2}\right)\frac{\dd^n}{\dd x^n}\left(\exp\left(-\frac{x^2}{2}\right)\right)\text{ for } n\geq 1\label{eq:hermite}.
\end{align}
The $n$-th Wiener chaos $\mathcal H_n$ is the closed linear subspace of $L^2(\Omega, \mathcal F, \mathbb P)$ generated by the family of random variables $\left\{H_n\left(\int_0^T h_s \dd W_s\right) : \|h\|_{L^2([0,T])}=1\right\}$. The vector spaces $\mathcal H_n,n\geq 0$, are orthogonal, giving rise to the Wiener chaos expansion~\cite[Theorem~1.1.1]{MR2200233}
\begin{equation}
    L^2(\Omega, \mathcal F, \mathbb P) = \bigoplus_{n=0}^\infty \mathcal H_n.
\end{equation}
Hence, any random variable $Y \in L^2(\Omega, \mathcal F, \mathbb P)$ admits an orthogonal decomposition
\begin{equation}
    Y = y_0 + \sum_{k=1}^\infty \sum_{|n|=k} y_k^n \prod_{i=1}^\infty H_{n_i}\left(\int_0^T e_i(s)\dd W_s\right),\label{eq:chaos_expansion}
\end{equation}
where $n=(n_i)_{i \geq 1}$ is a sequence of positive integers determining the polynomial degree, $|n|=\sum_{i \geq 1}n_i$, and $(e_i)_{i \geq 1}$ is an orthonormal basis of $L^2([0,T])$.
The coefficients are given by projection,
\begin{align*}
    y_0 = \E[Y],\quad
    y_k^n = \E\left[Y \prod_{i=1}^\infty H_{n_i}\left(\int_0^T e_i(s)\dd W_s\right)\right].
\end{align*}

Next, we consider the one-dimensional continuous stochastic process, $(X_t)_{t \in [0,T]}$ satisfying \cref{eq:sde}.
Using \cref{eq:chaos_expansion} and Malliavin calculus, the following propagator system was derived in~\cite{Huschto2014}.

\begin{theorem}[{Propagator system \cite[Theorem~2]{Huschto2014}}]\label{theorem:propagator_system}
    Given \cref{assump:uniqueness_sde}, let $(X_t)_{t \in [0,T]}$ satisfy \cref{eq:sde} and assume that $(X_t)_t \in L^2([0,T]\times \Omega)$.
    Then, $X_t$ exhibits the chaos expansion 
    \begin{equation}
        X_t = \sum_{k \geq 0}\sum_{|\alpha|=k} x_\alpha(t) \underbrace{\prod_{i=1}^\infty H_{\alpha_i}\overbrace{\left(\int_0^T e_i(s)\dd W_s\right)}^{=:G_i}}_{=:\Psi_\alpha}
        \label{eq:sde_chaos}
    \end{equation}
    and the coefficients $x_\alpha(t)$ satisfy the system of ordinary differential equations (ODE)
    \begin{align}
        \frac{\dd x_\alpha}{\dd t}(t) &= \mu(t, X_t)_\alpha + \sum_{j=1}^\infty \sqrt{\alpha_j} e_j(t)\sigma(t, X_t)_{\alpha^-(j)}, \label{eq:time_ode_coeff}\\
        x_\alpha(0) &= 1_{\alpha=0}x_0, \label{eq:ic_ode_coeff}
    \end{align}
    where $\mu(t, X_t)_\alpha$ (resp. $\sigma(t, X_t)_\alpha$) denotes the $\alpha$-coefficients of the Wiener chaos expansion associated with the random variable $\mu(t,X_t)$ (resp. $\sigma(t,X_t)$), and $\alpha^-(j)=(\alpha_1,\dots,\alpha_{j-1},\alpha_j-1,\alpha_{j+1},\dots)$.
\end{theorem}

In order to approximate the process $(X_t)_t$, one considers the truncation,
\begin{equation}
    X_t^{p,k} := \sum_{j=0}^p \sum_{|\alpha|=j} x_\alpha(t) \prod_{i=1}^k H_{\alpha_i}\left(\int_0^T e_i(s)\dd W_s\right), \label{eq:chaos_expansion_truncatation}
\end{equation}
which uses orthogonal projections $\Psi_\alpha$ only with respect to the first $k$ basis elements $(e_i)_{1 \leq i \leq k}$ and only up to the $p$-th order Wiener chaos.
In \cite{huschtoAsymptoticErrorChaos2019}, the authors derive an upper bound on the $L^2$-error $\E[(X_t^{p,k}-X_t)^2]$.

\begin{theorem}[{$L^2$-error of the Wiener chaos truncation}]\label{theorem:l2_error_truncation}
    Given \cref{assump:uniqueness_sde} and let $(X_t)_{t \in [0,T]}$ satisfy \cref{eq:sde}.
    Moreover, assume that $\mu,\sigma \in C^{1,\infty}([0,T]\times \R)$ such that
    \begin{equation*}
        \left|\frac{\partial^{\ell+m}\mu}{\partial t^\ell \partial x^m}(t, x) - \frac{\partial^{\ell+m}\mu}{\partial t^\ell \partial x^m}(t, y) \right| + \left|\frac{\partial^{\ell+m}\sigma}{\partial t^\ell \partial x^m}(t, x) - \frac{\partial^{\ell+m}\sigma}{\partial t^\ell \partial x^m}(t, y) \right| \leq K|x-y|,\quad K>0,
    \end{equation*}
    for $t \in [0,T],\ x,y \in \R$, any $\ell \in \{0,1\}$, and $m \geq 0$.
    Then it holds that
    \begin{equation}
        \E[(X_t^{p,k} - X_t)^2] \leq C(1+x_0^2)\left(\frac{1}{(p+1)!} + \sum_{\ell=k+1}^\infty \left(E_\ell(t)^2 + \int_0^t E_\ell^2(\tau)\dd\tau\right) \right), \label{eq:l2_error_truncation}
    \end{equation}
    where $C = C(t, K)$ is a positive constant and the function $E_\ell(t)$ is defined by
    \begin{equation}
        E_\ell(t) := \int_0^t e_\ell(s)\dd s. \label{eq:int_basis}
    \end{equation}
\end{theorem}
%
In the following, we use the Haar basis $\left\{e_0, e_{2^{n-1}+j} : 1 \leq j \leq 2^{n-1}, n \geq 1\right\}$ defined by
\begin{align*}
    e_0(t) := \frac{\mathbf 1_{[0, T]}(t)}{\sqrt T},\quad
    e_{2^{n-1}+j}(t) := \sqrt{\frac{2^{n-1}}{T}}\left(\mathbf 1_{\left[T\frac{2j-2}{2^n}, T\frac{2j-1}{2^n}\right]}-\mathbf 1_{\left[T\frac{2j-1}{2^n}, T\frac{2j}{2^n}\right]}\right).
\end{align*}
For this basis, the integrals $G_i$ in~\cref{eq:sde_chaos} can be computed explicitly, which is used in our analysis.
In fact,
\begin{align*}
    \int_0^T e_0(t) \dd W_t &= \frac{1}{\sqrt T} W_T,\\
    \int_0^T e_{2^{n-1}+j}(t) \dd W_t &= \frac{2^{\frac{n-1}{2}}}{\sqrt T} \left(\left(W_{T\frac{2j-1}{2^n}} - W_{T\frac{2j-2}{2^n}}\right) - \left(W_{T\frac{2j}{2^n}} - W_{T\frac{2j-1}{2^n}}\right)\right).
\end{align*}
    \section{Deep operator networks}
\label{sec:DON}

In this section, we formally introduce neural networks and review results on deep operator networks~\cite{lanthalerErrorEstimatesDeepONets2022}.
These form the foundation of our NN architecture for SDEs that is presented in~\Cref{sec:sdeonet}.

\subsection{Neural Networks}
\label{sec:NN}

A general neural network definition is given as follows.
\begin{definition}[Neural network and $\sigma$-realisation]\label{def:nn}
    Let $d,s,L \in \N\setminus \{0\}$. A \emph{neural network} $\Phi$ with input dimension $d$, output dimension $s$ and $L$ layers is a sequence of matrix-vector tuples,
    \begin{equation*}
        \Phi = ((W^1, b^1),\dots,(W^L,b^L)),
    \end{equation*}
    where $W^\ell \in \R^{n_\ell \times n_{\ell-1}}$ and $b^\ell \in \R^{n_\ell}$, with $n_0=d$, $n_L=s$ and $n_1,\dots,n_{L-1}\in\N$.
    
    We denote by $\mathcal N_{d,L,s}$ the space of neural network $\Phi$ with input dimension $d$, output dimension $s$ and $L$ layers, and also $\mathcal N_{d,s} := \bigcup_{L \geq 1}\mathcal N_{d,L,s}$. If $K \subset \R^d$ and $\sigma : \R \to \R$ is an arbitrary activation function, the associated realisation of $\Phi$ with activation function $\sigma$ over $K$ is defined as the map $ R_\sigma \Phi : K \to \R^s$ such that
    \begin{equation*}
        R_\sigma \Phi(x) = A_L \circ \sigma \circ A_{L-1} \circ \dots \circ \sigma \circ A_1 (x),
    \end{equation*}
    where $A_\ell(x) = W^\ell \cdot x + b^\ell$ is an affine transformation.

    We also introduce the following nomenclature for a neural network $\Phi \in \mathcal N_{d,s}$,
    \begin{align*}
        \Depth(\Phi) := L,\qquad
        \Size(\Phi) := \|\Phi\|_0
        := \sum_{\ell=1}^L (n_{\ell-1} + 1)n_\ell.
    \end{align*}
\end{definition}

The parallelisation of neural networks as shown in \cref{fig:parallelisation_nn} is a common operation that we use in the subsequent analysis.
\begin{proposition}[{Parallelisation \cite[Definition~2.7]{Petersen2018}}]\label{proposition:parallelisation}
    Let $\Phi^1 = ((W_1^1,b_1^1),\dots,(W_{L}^1,b_{L}^1)) \in \mathcal N_{d,L,s_1}$ and $\Phi^2 = ((W_1^2,b_1^2),\dots,(W_{L}^2,b_{L}^2)) \in \mathcal N_{d,L,s_2}$ be two neural networks with $d$-dimensional input.
    Then $P(\Phi^1, \Phi^2) = ((\tilde A_1, \tilde b_1),\dots,(\tilde A_L, \tilde b_L))$, where
    \begin{equation*}
        \tilde A_1 := \begin{pmatrix}
            A_1^1\\
            A_1^2
        \end{pmatrix},\quad
        \tilde b_1 := \begin{pmatrix}
            b_1^1\\
            b_1^2
        \end{pmatrix},\quad
        \tilde A_\ell := \begin{pmatrix}
            A_\ell^1 & 0\\
            0 & A_\ell^2
        \end{pmatrix},\quad
        \tilde b_\ell := \begin{pmatrix}
            b_\ell^1\\
            b_\ell^2
        \end{pmatrix}
        \quad\text{for}~1 < \ell \leq L
    \end{equation*}
    is a neural network with $d$-dimensional input and $L$ layers, called the \emph{parallelisation} of $\Phi^1$ and $\Phi^2$. Moreover, $P(\Phi^1, \Phi^2)$ satisfies
    \begin{align*}
        \Size(P(\Phi^1, \Phi^2)) = \Size(\Phi^1) + \Size(\Phi^2)\quad\text{ and }\quad
        R_\sigma P(\Phi^1, \Phi^2) = (R_\sigma\Phi^1, R_\sigma \Phi^2).
    \end{align*}
\end{proposition}

\begin{figure}
    \centering
    \includegraphics{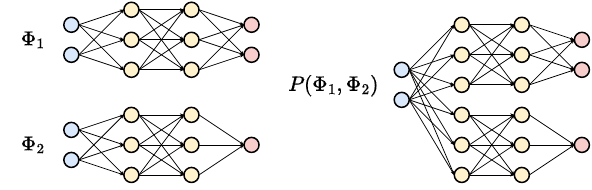}
    \caption{\textbf{Left:} Two neural networks. \textbf{Right:} Parallelisation with shared inputs of both networks according to \cref{proposition:parallelisation}.}
    \label{fig:parallelisation_nn}
\end{figure}
We recall approximation results for continuously differentiable and H\"older continuous functions that are used in the later analysis.
The neural network complexity required for approximating multivariate Hermite polynomials up to a certain accuracy has been examined in~\cite{Schwab2023}.
\begin{theorem}[{Deep ReLU neural networks approximation of multivariate Hermite polynomials \cite[Theorem~3.7]{Schwab2023}}]\label{theorem:nn_approx_hermite}
    Let $\Lambda \subset \{\alpha \in \N^\infty : |\alpha| < \infty\}$ be \emph{finite}.
    For every $\varepsilon \in (0, e^{-1})$ there exists a neural network $\Phi_\varepsilon$ such that
    \begin{equation*}
        \max_{\alpha \in \Lambda}\|H_\alpha - \tilde H_{\varepsilon,\alpha}\|_{L_\mu^2(\R^{|\supp(\Lambda)|})} \leq \varepsilon,
    \end{equation*}
    where $H_{\varepsilon,\alpha} = R_\sigma \Phi_\varepsilon : \R^{|\supp(\Lambda)|} \to \R^{|\Lambda|}$, $\supp(\Lambda) := \{j \in \supp(\alpha) : \alpha \in \Lambda\}$ and $\mu$ is the multivariate Gaussian measure.
    Moreover, there exists a positive constant $C$ (independent of $m(\Lambda) := \max_{\alpha \in \Lambda}|\alpha|$, $d(\Lambda) := \max_{\alpha \in \Lambda}|\alpha|_0$ and of $\varepsilon$) such that
    \begin{align*}
        \Size(\Phi_\varepsilon) &\leq C|\Lambda| m(\Lambda)^3 \log(1+m(\Lambda)) d(\Lambda)^2 \log(\varepsilon^{-1}),\\
        \Depth(\Phi_\varepsilon) &\leq C m(\Lambda) \log(1+m(\Lambda))^2 d(\Lambda) \log(1+d(\Lambda))\log(\varepsilon^{-1}).
    \end{align*}
\end{theorem}

\begin{remark}
    The result above gives complexity results for approximating multivariate Hermite polynomials using Deep ReLU neural networks.
    Instead, if we use the Rectified Power Unit (RePU) activation function defined by $\operatorname{RePU}^p : x \mapsto \max(x,0)^p$, with $p \geq 2$, then there exists a Deep Neural network with this activation function which represents \textit{exactly} a multivariate polynomial.
    From \cite[Proposition~2.14]{Opschoor2021} we have that there exists a Deep RePU Neural network $\phi$ which represents exactly any polynomial $p \in \mathbb P_\Lambda$, and such that
    \begin{align*}
        \Size(\phi) \leq C |\Lambda|,\quad
        \Depth(\phi) \leq C\log(|\Lambda|),
    \end{align*}
    with a constant $C > 0$ depending only on the power $p$.
\end{remark}

\begin{theorem}[{Approximation of $\beta$-H\"older continuous function \cite[Theorem~3.1]{Petersen2018}}]\label{theorem:approx_holder_nn}
    Let $d \in \N$, $B,p > 0$ and $\beta=(n,\xi) \in \N \times (0,1]$.
    Then, there exists a constant $c = c(d,n,\xi,B) > 0$ such that for any function $f \in C^\beta([-1/2,1/2]^d)$ with $\|f\|_{C^\beta} \leq B$ and any $\varepsilon \in (0,1/2)$ there is a neural network $\Phi_\varepsilon^f$ such that
    \begin{align*}
        \|R_\sigma(\Phi_\varepsilon^f) - f\|_{L^p([-1/2,1/2]^d)} < \varepsilon,\quad
        \|R_\sigma(\Phi_\varepsilon^f)\|_\infty \leq \lceil B \rceil,
    \end{align*}
    and
    \begin{align*}
        \Depth(\Phi_\varepsilon^f) \leq (2 + \lceil \log_2(n + \xi)\rceil)\left(11+\frac{n+\xi}{d}\right),\quad
        \Size(\Phi_\varepsilon^f) \leq c\varepsilon^{-\frac{d}{n+\xi}}.
    \end{align*}
\end{theorem}

\subsection{Operator Neural Networks}
\label{sec:operator networks}

The basis for our NN construction are recent results on operator networks.
Operators are mappings between infinite dimensional function spaces.
Prominent examples are the solution operators for ODEs and PDEs, which map function space inputs to the solution of the differential equation in another function space \cite{tianpingchenUniversalApproximationNonlinear1995,luDeepONetLearningNonlinear2021,li2021fourier,kovachki2023neural}.
Typical inputs are parameters describing coefficients or initial and boundary data, which in particular is common in Uncertainty Quantification.
The differential equation setting also provides a mathematical framework for statistical inverse problems, where the object of interest is the inverse operator that maps some observables to the underlying model data that is to be inferred~\cite{stuart2010inverse,gao2023adaptive}.
In recent years, machine learning-based operator approximation has attracted growing interest due to the possibly high cost of classical operator approximation techniques, particularly those related to high-dimensional parametric and nonlinear PDEs~\cite{Marcati2023,Deng2022}.
Opposite to simulation methods, operator learning infers operators from solution data and a well-known approach is given by the DeepONet architecture~\cite{luDeepONetLearningNonlinear2021}.
DeepONet can query any coordinate in the (parameter) domain to obtain the value of the output function.
However, for training and testing, the input function must be evaluated at a set of predetermined locations (often called ``snapshots'' in reduced basis methods), which requires a fixed observation grid for all observations.

\begin{remark}
For illustration, consider the 1D dynamical system defined on domain $[0,1]$ by
\begin{equation*}
    \frac{\dd u}{\dd x}(x) = f(u(x), \phi(x), x),\qquad u(0)=0.
\end{equation*}
The \emph{operator} $\mathcal G$ that maps the perturbation $\phi$ to the solution $u$ satisfies
\begin{equation*}
    (\mathcal G \phi)(x) = \int_0^x f((\mathcal G \phi)(y), \phi(y), y)\dd y.
\end{equation*}
In the linear case $f(u(x), \phi(x), x) = \phi(x)$, the considered operator to be learned is the \emph{antiderivative operator}
\begin{equation*}
    \mathcal G : \phi \mapsto \left(u: x \mapsto \int_0^x \phi(y)\dd y\right).
\end{equation*}
\end{remark}

\begin{figure}
    \centering
    \begin{tikzpicture}
        \matrix (m) [matrix of math nodes,row sep=8em,column sep=8em,minimum width=2em]
        {
        X & Y\\
        \R^m & \R^p\\};
        \path[-stealth]
            (m-1-1) edge [dashed] node [left, align=center] {Encoder\\$\mathcal E : u \mapsto \mathbf u := (u(x_1),\dots,u(x_m))$} (m-2-1)
                    edge node [above] {$\mathcal G$} (m-1-2)
            (m-2-1.east|-m-2-2) edge [dashed] node [below, align=center] {Approximator\\$\mathcal A : \mathbf u \mapsto (\beta_1(u),\dots,\beta_p(u))$} (m-2-2)
            (m-2-2) edge [dashed] node [right, align=center] {Reconstructor\\ $\mathcal R : \beta \mapsto \sum_{i=1}^p \beta_i(u) \tau_i$} (m-1-2);
    \end{tikzpicture}
    \caption{Structure of the deep operator network.}
    \label{fig:DeepONet_diagram}
\end{figure}
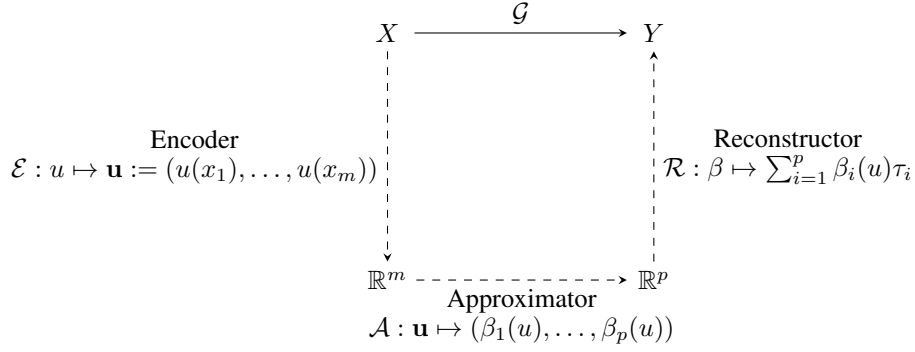

We henceforth assume that $D \subset \R^d$ and $U \subset \R^n$ are compact domains (e.g. with Lipschitz boundary).
The architecture of a deep operator network is depicted in~\cref{fig:DeepONet_diagram}, where the operator $\mathcal G$ is to be represented approximately.
For this, the encoder $\mathcal E$ results in a finite dimensional input representation that is mapped via the approximator $\mathcal A$ to the finite dimensional output.
Eventually, this is transferred to the image of $\mathcal G$ by the reconstruction $\mathcal R$.
The following definition makes this procedure rigorous.

\begin{definition}[{Deep operator network (DeepONet) \cite[Definitions~2.1 \& 2.4]{lanthalerErrorEstimatesDeepONets2022}}]\label{def:deeponets}
    Assume separable Banach spaces $X,Y$ with continuous embeddings $\iota : X \hookrightarrow L^2(D)$ and $\bar\iota : Y \hookrightarrow L^2(U)$.
    Let $\mu \in \mathcal P_2(X)$ be a Borel probability measure on $X$ with finite second moments such that there exists $A \subset X$ with $\mu(A)=1$ and $A$ consists of continuous functions.
    Moreover, let $\mathcal G : X \to Y$ be a Borel measurable mapping such that $\mathcal G \in L_\mu^2$.
    For the construction of the DeepONet architecture depicted in \cref{fig:DeepONet_diagram}, three \emph{operators} are used:
    \begin{itemize}
        \item \textbf{Encoder}: Given a set of \emph{sensor} points $\{x_j\}_{j=1}^m\subset X$, define the linear mapping
        \begin{equation}
            \mathcal E : \left\{
            \begin{array}{rl}
                C(D) & \longrightarrow \R^m \\
                u & \longmapsto (u(x_1), \dots, u(x_m))
            \end{array}
            \right.
            \label{eq:deeponet_encoder}
        \end{equation}
        as the \emph{encoder} mapping.
        \item \textbf{Approximator}: Given sensor points $\{x_j\}_{j=1}^m$, the \emph{approximator} is a deep neural network $\mathcal A \in \mathcal N_{m,p}$.
        \item Given the encoder and approximator, we define the \textbf{branch net} \begin{equation}
            \beta : \left\{
            \begin{array}{rl}
                C(D) & \longrightarrow \R^p \\
                u & \longmapsto R_\sigma \mathcal A \circ \mathcal E(u)
            \end{array}
            \right..
            \label{eq:deeponet_branch}
        \end{equation}
        It represents the coefficients in the basis expansion.
        \item Denote a \textbf{trunk net} by $\tau \in \mathcal N_{n,p}$ as a deep neural network representation of basis functions based on the encoder data.
        \item \textbf{Reconstructor}. The $\tau$-induced \emph{reconstructor} is given by
        \begin{equation}
            \mathcal R : \left\{
            \begin{array}{rl}
                \R^p & \longrightarrow C(U) \\
                \{\alpha_k\}_{k=1}^p & \longmapsto \sum_{k=1}^p \alpha_k (R_\sigma \tau)_k
            \end{array}
            \right. .
            \label{eq:deeponet_reconstructor}
        \end{equation}
    \end{itemize}

    A \textbf{DeepONet} $\mathcal N$ approximates the nonlinear operator $\mathcal G$.
    It is a mapping $\mathcal N : C(D) \to L^2(U)$ of the form $\mathcal N = \mathcal R \circ R_\sigma \mathcal A \circ \mathcal E$, where $\mathcal E : (X, \|\cdot\|_X) \to (\R^m, \|\cdot\|_{\ell^2})$ denotes the encoder $\mathcal E$ given by \cref{eq:deeponet_encoder}, $R_\sigma \mathcal A : (\R^m, \|\cdot\|_{\ell^2}) \to (\R^p, \|\cdot\|_{\ell^2})$ denotes the approximation network, and $\mathcal R : (\R^p, \|\cdot\|_{\ell^2}) \to (L^2(U), \|\cdot\|_{L^2(U)})$ denotes the reconstruction of the form \cref{eq:deeponet_reconstructor}, induced by the trunk net $\tau$.
\end{definition}

In \cite{lanthalerErrorEstimatesDeepONets2022}, the authors study the approximation of $\mathcal G$ by $\mathcal N$.
For the analysis, they consider the following error, measured in $L_\mu^2$, where $\mu$ is associated with $X$,
\begin{equation}
\label{eq:deeponet_error}
    \begin{aligned}
        \hat E^2 := \|\mathcal G - \mathcal N\|_{L_\mu^2}
        := \int_X \int_U |\mathcal G(u)(y) - \mathcal N(u)(y)|^2 \dd y \dd\mu(u).
    \end{aligned}
\end{equation}

\subsection{SDEONet architecture}
\label{sec:sdeonet}

We now describe the construction of the SDEONet architecture, which is inspired by the DeepONet presented above and combined with the chaos representation of~\Cref{sec:wiener chaos}.
For this, recall the polynomial chaos expansion~\cref{eq:sde_chaos} for a stochastic process written as
\begin{equation}
    X_t(\omega) = \sum_{k \geq 0}\sum_{|\alpha|=k} x_\alpha(t) \Psi_\alpha(\omega), \label{eq:sde_params}
\end{equation}
where $\omega \in \Omega$.
One can hence define a nonlinear operator $\mathcal G$ such that $X_t(\omega) = \mathcal G(\{W_s(\omega)\}_{s \in [0,T]}\})(t)$, i.e., it maps the Brownian motion $W$ to the continuous stochastic process $(X_t)_{t \in [0,T]}$ satisfying \cref{eq:sde}.
One can intuitively approximate such an operator $\mathcal G$ with a DeepONet $\mathcal N$.

\begin{definition}[SDE solution operator]\label{def:sde_sol_operator}
    A (nonlinear) operator $\mathcal G : L^2([0,T] \times \Omega) \to L^2([0,T] \times \Omega)$ is said to be a \emph{solution operator} if $\mathcal G(W)$ is a continuous stochastic process satisfying \cref{eq:sde}, where $W = (W_t)_{t \in [0,T]}$ is a Brownian motion.
    Given \cref{assump:uniqueness_sde}, it is the strong solution of the SDE with respect to $W$.
\end{definition}

\begin{example}[Operator $\mathcal G$ of linear SDE]
    Consider the following \emph{linear SDE}
    \begin{equation*}
        \dd X_t = \left(a(t)X_t + b(t)\right)\dd t + h(t)\dd W_t,
    \end{equation*}
    where $a,b$ and $h$ are bounded functions on $[0,T]$.
    Then, for $t \in [0,T]$ we have that
    \begin{align*}
        \mathcal G(W)(t) = \exp\left(\int_0^t a(s)\dd s\right) & \left( X_0 + \int_0^t\exp\left(-\int_0^s a(r)\dd r\right)b(s)\dd s \right. \\
        &\left. + \int_0^t\exp\left(-\int_0^s a(r)\dd r\right)h(s)\dd W_s\right).
    \end{align*}
\end{example}

The ingredients of the SDEONet architecture are defined next.
\begin{definition}[SDEONet]\label{def:sde_deeponet}
    Let $W = (W_t)_{t \in [0,T]}$ be a Brownian motion and $p,m=2^k \in \N$ polynomial chaos discretisation parameters.
    We construct the SDEONet as a composition of the following operators:
    \begin{itemize}
        \item \textbf{Encoder}: The mapping
        \begin{equation}
            \mathcal E^{p,m} :
            \left\{
            \begin{aligned}
                L^2([0,T] \times \Omega) &\longrightarrow L^2(\Omega, \R)^{m}\\
                W &\longmapsto \left(G_i := \int_0^T e_i(t)\dd W_t\right)_{i=0}^{m-1}
            \end{aligned}
            \right.
            \label{eq:sde_deeponet_encoder}
        \end{equation}
        maps the Brownian motion $W$ to the $G_i$,
        for $i = 2^{n-1} + j$, $1 \leq j \leq 2^{n-1}$ and $1 \leq n \leq k$.
        \item \textbf{Approximator}: Given the values $\{G_i\}_{i=0}^{m-1}$, we denote an \emph{approximator} $\mathcal A \in \mathcal N_{m,p}$ as a deep neural network such that its $\sigma$-realisation $R_\sigma \mathcal A = (\widetilde{\Psi_j})_{j=1}^p : L^2(\Omega)^m \to L^2(\Omega, \R)^p$ approximates the \emph{chaos polynomials} $\Psi_{k_j^*}$.
        \item \textbf{Branch net}. Given the \emph{encoder} and \emph{approximator}, the \emph{branch net} is defined by $\beta := \mathcal A \circ \mathcal E$.
        \item \textbf{Trunk net}: We denote a \emph{trunk net} $\tau^p \in \mathcal N_{1,p}$ as a deep neural network such that its $\sigma$-realisation $R_\sigma\tau^p = (\widetilde{x_j})_{j=1}^p : [0,T] \to \R^{p}$ approximates the coefficient functions $x_{k_j^*}$ in \cref{theorem:propagator_system}.
        \item \textbf{Reconstructor}. The $\tau$-induced \emph{reconstructor} is given by,
        \begin{equation}
            \mathcal R_\tau^p : \left\{
            \begin{array}{rl}
                 L^2(\Omega, \R)^p &\longrightarrow L^2([0,T] \times \Omega) \\
                 (\widetilde{\Psi_j})_{j=1}^p & \longmapsto \sum_{j=1}^p \widetilde{x_j}\widetilde{\Psi_j}
            \end{array}.
            \right.
            \label{eq:sde_deeponet_reconstructor}
        \end{equation}
        This is a mapping that approximates $(X_t^{{m,p}^*})_{t \in [0,T]}$.
    \end{itemize}
    A \textbf{SDEONet} $\mathcal N^{m,p}$ approximates the nonlinear operator $\mathcal G$ in \cref{def:sde_sol_operator}.
    It is defined as mapping $\mathcal N^{m,p} : L^2([0,T] \times \Omega) \to L^2([0,T] \times \Omega)$ of the form $\mathcal N^{m,p} = \mathcal R_{\tau^p} \circ \mathcal A \circ \mathcal E$.
\end{definition}
    \section{Convergence analysis}
\label{sec:analysis}

In this section, a complete error analysis for the SDEONet architecture described in the last section is carried out.
Similar to the analysis of DeepONets in~\cite{lanthalerErrorEstimatesDeepONets2022}, the overall error is split into several components that are examined successively, namely truncation, approximation, and reconstruction errors.
We state our main convergence result as a combination of the subsequent estimates in the following~\cref{theorem:main_result}.

Let $\mathcal G : L^2([0,T]\times \Omega) \to L^2([0,T]\times\Omega)$ be a SDE solution operator according to \cref{def:sde_sol_operator} and $\mathcal N^{m,p} $ be a SDEONet as of \cref{def:sde_deeponet} with $m,p \in \N$.
We consider the error measured in the $L^2([0,T]\times\Omega)$-norm defined by
\begin{equation}
    \begin{aligned}
        \hat E^2 &:= \int_0^T \E[|\mathcal G(W)(t) - \mathcal N(W)(t)|^2] \dd t = \int_0^T \E[|X_t - \widetilde X_t^{m,p}|^2]\dd t.
    \end{aligned}
    \label{eq:operator_error}
\end{equation}

\subsection{Auxiliary results}
\label{sec:auxiliary}

To prepare the convergence analysis in the following sections, we first introduce a decomposition of the error.
This requires the finite multi-index set
\begin{equation}
    \mathcal I_{p,k} := \left\{\alpha \in \N^k : |\alpha| := \sum_{i=1}^k \alpha_i \leq p \right\}
    \label{eq:finite_multiindex}
\end{equation}
and the notion of \emph{best $p$-terms} given by
\begin{equation}
    k^* := \argmin_{k = (k_1,\dots,k_p) \in \mathcal I^p} \int_0^T \E[|X_t - \sum_{j=1}^p x_{k_j}(t) \Psi_{k_j}(W)|^2]\dd t.
    \label{eq:best_p_terms}
\end{equation}
Additionally, we define the following set of tuples
\begin{equation*}
    J_p := \left\{(m,n) \in \N^2 : \binom{m+n}{n} \leq p\right\}.
\end{equation*}

The truncation error with respect to the basis $(e_i)_i$ in $L^2([0,T])$ can be bounded by the next result.
\begin{proposition}\label{proposition:decay_rest}
    Let $n \in \N$.
    Then, for all $t \in [0,T]$,
    \begin{equation}
    \label{eq:gi_haar}
        E_{2^{n-1}+j}(t) = \begin{cases}
            \frac{2^{\frac{n-1}{2}}}{\sqrt T}\left(t-T\frac{2j-2}{2^n}\right) & T\frac{2j-2}{2^n} \leq t \leq T\frac{2j-1}{2^n}\\
            \frac{2^{\frac{n-1}{2}}}{\sqrt T}\left(T\frac{2j}{2^n} - t\right) & T\frac{2j-1}{2^n} \leq t \leq T\frac{2j}{2^n}\\
            0 & \emph{else}
        \end{cases},
    \end{equation}
    and, with $E$ defined in \cref{eq:int_basis},
    \begin{equation}
    \label{eq:remainder_basis}
        \sum_{\ell=n+1}^\infty \sum_{j=1}^{2^n} \left(E_{2^{\ell-1}+j}^2(t) + \int_0^t E_{2^{\ell-1}+j}^2(\tau)\dd\tau \right) \leq 2T(1+t)2^{-n}.
    \end{equation}
\end{proposition}
\begin{proof}
    The first expression \cref{eq:gi_haar} follows from the definition of $e_{2^{n-1}+j}$.
    Concerning the second expression \cref{eq:remainder_basis}, first, note that $\max_{t \in [0,T]} E_{2^{n-1}+j}^2(t) = T2^{-(n+1)}$ and that $E_{2^{n-1}+j}(t) \neq 0$ only for $t \in \left[T\frac{2j-2}{2^n}, T\frac{2j}{2^n}\right]$. Then,
    \begin{align*}
        \sum_{\ell=n+1}^\infty \sum_{j=1}^{2^n} \left(E_{2^{\ell-1}+j}^2(t) + \int_0^t E_{2^{\ell-1}+j}^2(\tau)\dd\tau \right) & \leq \sum_{\ell=n+1}^\infty \left(T2^{-(\ell+1)} + tT 2^{-(\ell+1)}\right)\\
        & = 2T(1+t)2^{-n},
    \end{align*}
    since it is a geometric series.
\end{proof}

The subsequent result permits the bounding of \cref{eq:operator_error} by three additional terms through the application of the triangle inequality.
\begin{lemma}[Decomposition of the error]\label{lemma:decomposition_error}
    Let $p,m \in \N$, $\mathcal G$ be a SDE solution operator according to \cref{def:sde_sol_operator} and $\mathcal N^{p,m}$ a \emph{SDEONet}.
    Then, the error $\hat E$ \cref{eq:operator_error} can be decomposed as
    \begin{equation}
    \label{eq:full_error}
        \hat E \leq \hat E_\text{Trunc} + \hat E_\text{Approx} + \hat E_\text{Recon},
    \end{equation}
    with
    \begin{subequations}
    \begin{align}
        \hat E_\text{Trunc} &:= \left(\int_0^T \E[|X_t - \sum_{j=1}^p x_{k_j^*}(t)\Psi_{k_j^*}|^2]\dd t\right)^{1/2},\label{eq:trunc_error}\\
        \hat E_\text{Approx} &:= \left(\int_0^T \E[|\sum_{j=1}^p x_{k_j^*}(t)\Psi_{k_j^*} - \sum_{j=1}^p x_{k_j^*}(t)\widetilde{\Psi_j}|^2]\dd t\right)^{1/2},\label{eq:approx_error}\\
        \hat E_\text{Recon} &:= \left(\int_0^T \E[|\sum_{j=1}^p x_{k_j^*}(t)\widetilde{\Psi_j} - \sum_{j=1}^p \widetilde{x_j}(t)\widetilde{\Psi_j}|^2]\dd t\right)^{1/2}.\label{eq:reconstruct_error}
    \end{align}
    \end{subequations}
\end{lemma}

\subsection{Truncation error}
\label{sec:truncation}
We first estimate the truncation error, which results from only using a finite number of basis elements and polynomials in the Wiener chaos expansion.

\begin{lemma}[{Upper bound of truncation error}]\label{lemma:ub_trunc_error}
    With \cref{assump:uniqueness_sde} and let $\mu,\sigma$ satisfy the assumptions of \cref{theorem:l2_error_truncation}. 
    Then, the \emph{truncation error} \cref{eq:trunc_error} satisfies
    \begin{equation}
        \hat E_\text{Trunc} \leq \min_{m,n \in J_p}\left( (1+x_0^2)\left(\int_0^T C(t, K)\left(\frac{1}{(m+1)!} + \frac{2T(1+t)}{n}\right) \dd t\right)\right)^{1/2}, 
    \end{equation}
    with $C(t,K)$ defined in \cref{theorem:l2_error_truncation}.
\end{lemma}

Note that the above lemma shows that the truncation error decays factorially fast in the number of polynomial chaos terms and linearly in the number of basis elements.

\subsection{Approximation error}
\label{sec:approximation error}
The second term in~\cref{eq:full_error} is the approximation error that comes from the approximation of the chaos polynomials.
This term is more involved and requires the use of Malliavin calculus.
The approach is to explicitly introduce the $L^2$ error of the polynomial chaos and show that a neural network can indeed approximate the Hermite polynomials.

For the upper bound of the approximation error, we recall the following results.

\begin{theorem}[{\cite[Theorem~2.9, p. 289]{Karatzas1998}}]\label{theorem:square_int_process}
    Suppose that \cref{assump:uniqueness_sde} is fulfilled. Then, there exists a continuous adapted process $X$, which is a strong solution of \cref{eq:sde} with respect to $W$ with initial condition $x_0$. Moreover, this process is square-integrable: for every $T > 0$ there exists a constant $C := C(K,T)$ such that
    \begin{equation*}
        \E[\|X_t\|_2^2] \leq C(1+\E[\|x_0\|_2^2])\exp(Ct)\quad\text{ for } 0 \leq t \leq T.
    \end{equation*}
\end{theorem}

\begin{proposition}[{\cite[Proposition~4.1]{huschtoAsymptoticErrorChaos2019}}]\label{proposition:malliavin_bound}
    Under the conditions of \cref{theorem:l2_error_truncation}, we obtain the estimate
    \begin{equation*}
        \E[(D_{s_1,\dots,s_n}^n X_t)^2] \leq C^n (1+x_0^2)\exp(Cnt),
    \end{equation*}
    where $C$ is the same as in \cref{theorem:square_int_process}.
\end{proposition}

\begin{theorem}[{\cite{Stroock1987}}]\label{theorem:kernel_stroock}
    Let $F \in L^2(\Omega)$. Suppose that $F$ is \emph{infinitely Malliavin differentiable} and that for every $k \geq 0$ the \emph{$k$-th Malliavin derivative $D^k F$ of $F$ is square-integrable}. Then the symmetric functions $f_n$ in the chaos decomposition
    \begin{equation*}
        F = \sum_{n=0}^\infty I_n(f_n)
    \end{equation*}
    can be computed by
    \begin{equation*}
        f_n = \frac{1}{n!} \E[D^n F].
    \end{equation*}
\end{theorem}

These results enable to bound the coefficient functions of the PCE \cref{eq:sde_params}.
\begin{lemma}\label{lemma:decay_squared_coef}
    Let \cref{assump:uniqueness_sde} be satisfied, and let $\mu,\sigma$ satisfy the assumptions of \cref{theorem:l2_error_truncation}.
    Consider the Wiener chaos expansion~\cref{eq:sde_params} of $X_t$
    \begin{equation*}
        X_t = \sum_{n=0}^\infty \sum_{|\alpha|=n} x_\alpha(t) \Psi_\alpha.
    \end{equation*}
    Then, for $m \in \mathbb N$ we have
    \begin{equation*}
        \sum_{\ell = 0}^m \sum_{|\alpha|=\ell} x_\alpha(t)^2 \leq (1+x_0^2)\left(e^{Ct e^{Ct}} - \frac{(Cte^{Ct})^{m+1}}{(m+1)!}\right),
    \end{equation*}
    where $t \in [0, T]$, $C = C(K, T)$ is a constant that depends only on $T$ and the regularity of $\mu$ and $\sigma$.
\end{lemma}

We can now bound the approximation error.
\begin{lemma}[{Upper bound of approximation error}]\label{lemma:ub_approx_error}
    Let \cref{assump:uniqueness_sde} be satisfied and let $\mu,\sigma$ satisfy the assumptions of \cref{theorem:l2_error_truncation}.
    Then, the \emph{approximation error}~\cref{eq:approx_error} satisfies
    \begin{align*}
        \hat E_\text{Approx}^2 &\leq \min_{(m,n) \in J_p} \E\left[\sum_{j=1}^p |\Psi_{k_j^*}-\widetilde{\Psi_j}|^2)\right]\\ &\times (1+x_0^2)\left(A(K,T) - \frac{1}{(m+1)!}\frac{T(CT)^{2(m+1)}}{2(m+1)+1}\left(1+\frac{1}{CT}\right)^{m+1}\right),
    \end{align*}
    with constants $A(K,T):=\int_0^T e^{Cte^{Ct}} \dd t$ and $C=C(K,T)$.
\end{lemma}

The approximation of the $\Psi_j$ by neural networks according to~\cref{theorem:nn_approx_hermite} yields the following result.
\begin{corollary}[{Deep ReLU neural networks approximation of polynomials chaos}]\label{corollary:relu_approx_chaos}
    Let $p \in \N$. For any $\varepsilon \in (0,e^{-1}/p)$ there exists a neural network $\Phi_\varepsilon$ such that
    \begin{equation*}
        \E\left[\sum_{j=1}^p |\Psi_{k_j^*} - (R_\sigma\Phi_\varepsilon)_j|^2 \right] \leq \varepsilon
    \end{equation*}
    with complexity given by
    \begin{align*}
        \Size(\Phi_\varepsilon) &\leq Cp \max_{j \in \{1,\dots,p\}}|k_j^*|^3 \log(1+|k_j^*|)|k_j^*|_0^2 \log(p\varepsilon^{-1}),\\
        \Depth(\Phi_\varepsilon) &\leq C\max_{j \in \{1,\dots,p\}} |k_j^*|\log(1+|k_j^*|)^2|k_j^*|_0\log(1+|k_j^*|_0)\log(p\varepsilon^{-1}).
    \end{align*}
    Moreover, given such a neural network $\Phi_\varepsilon$, we have the bound
    \begin{equation*}
        \sum_{j=1}^p \E[|(R_\sigma\Phi_\varepsilon)_j|^2] \leq 2(\varepsilon + p).
    \end{equation*}
\end{corollary}
\begin{proof}
    Recall that $k^* := \argmin_{k = (k_1,\dots,k_p) \in \mathcal I^p} \int_0^T \E[|X_t - \sum_j x_{k_j}(t) \Psi_{k_j}(W)|^2]\dd t$ and let $j \in \{0, \dots, p\}$.
    By \cref{theorem:nn_approx_hermite} (take $\Lambda = \{k_j^*\}$) there exists a neural network $\Phi_{\varepsilon,j}$ such that
    \begin{equation*}
        \|\Psi_{k_j^*} - R_\sigma\Phi_{\varepsilon,j}\|_{L_\mu^2(\R^{|k_j^*|_0})} \leq \frac \varepsilon p.
    \end{equation*}
    Moreover, there exists a positive constant $C$ (independent of $|k_j^*|$, $|k_j^*|_0$, $\varepsilon$ and of $p$) such that
    \begin{align*}
        \Size(\Phi_{\varepsilon,j}) &\leq C |k_j^*|^3 \log(1+|k_j^*|)|k_j^*|_0^2 \log(p\varepsilon^{-1}),\\
        \Depth(\Phi_{\varepsilon,j}) &\leq C|k_j^*|\log(1+|k_j^*|)^2|k_j^*|_0\log(1+|k_j^*|_0)\log(p\varepsilon^{-1}).
    \end{align*}
    The result follows by parallelisation as in \cref{proposition:parallelisation}.

    Next, note that
    \begin{align*}
        \sum_{j=1}^p \E[|(R_\sigma\Phi_\varepsilon)_j|^2] &= \sum_{j=1}^p \E[|(R_\sigma\Phi_\varepsilon)_j - \Psi_{k_j^*} + \Psi_{k_j^*}|^2]\\
        &\leq 2\left(\sum_{j=1}^p \E[|(R_\sigma\Phi_\varepsilon)_j - \Psi_{k_j^*}|^2] + \sum_{j=1}^p \E[|\Psi_{k_j^*}|^2]\right).
    \end{align*}
    The first term is bounded by $\varepsilon$ due to the choice of $\Phi_\varepsilon$. Concerning the second term, the polynomials chaos are orthonormalised and hence $\E[|\Psi_{k_j^*}|^2] = 1$.
\end{proof}

\subsection{Reconstruction error}
\label{sec:reconstruction error}

For the following results, we consider the common activation function $\sigma=\ReLU$.
The last error term in~\cref{eq:full_error} is the approximation of the deterministic coefficient functions in the Wiener chaos expansion \cref{eq:sde_params}.
We show that they can be approximated by neural networks based on the regularity of the corresponding ODE trajectories and~\cref{theorem:approx_holder_nn}.
\begin{corollary}[{Approximation of ODEs}]\label{corollary:approx_1d_odes}
    Let $f : [t_0,t_1] \times \R^m \to \R^m \in C^k$ and Lipschitz with respect to the second variable, $k \in \N$, $p > 0$, and $x_0 \in \R^m$.
    Consider the Cauchy problem with $X : [t_0, t_1] \to \R^m$,
    \begin{equation}
        \frac{\dd X}{\dd t}(t) = f(t, X(t)),\quad X(t_0)=x_0.\label{eq:1d_ode}
    \end{equation}
    The ODE~\cref{eq:1d_ode} has a unique solution $X : [t_0,t_1] \to \R^m \in C^{k+1}$ and for any $\varepsilon \in (0,(m(t_1-t_0))^{1/p}/2)$ there exists a neural network $\Phi_\varepsilon^X \in \mathcal N_{1,m}$ such that
    \begin{align*}
        \|R_\sigma(\Phi_\varepsilon^X) - X\|_{L^p([t_0,t_1])} < \varepsilon,\quad
        \|R_\sigma(\Phi_\varepsilon^X)\|_\infty \leq K,
    \end{align*}
    where $K = K(t_0, t_1, X)$ is a constant, and
    \begin{align*}
        \Depth(\Phi_\varepsilon^X) &\leq (2+\lceil \log_2(k+1)\rceil)(12+k),\\
        \Size(\Phi_\varepsilon^X) &\leq mc\left(\frac{\varepsilon}{(t_1-t_0)^{1/p}}\right)^{-\frac{1}{k+1}}.
    \end{align*}
\end{corollary}

\begin{lemma}\label{lemma:nn_reconstruction}
    With \cref{assump:uniqueness_sde} and $n$-times continuously differentiable $\mu,\sigma,e_j$ with respect to their variables, for any $\varepsilon \in (0, \sqrt{pT}/2)$ there exists a neural network $\Phi_\varepsilon \in \mathcal N_{1,p}$ such that
    \begin{equation*}
        \sum_{j=1}^p \|x_{k_j^*}-\widetilde{x_j}\|_{L^2([0,T])}^2 \leq \varepsilon^2.
    \end{equation*}
    The \emph{reconstruction error} can be bounded like
    \begin{equation*}
        \hat E_\text{Recon}^2 \leq \varepsilon^2 \left(\sum_{j=1}^p \E[\widetilde{\Psi_j}^2]\right)
    \end{equation*}
    and
    \begin{align*}
        \Depth(\Phi_\varepsilon) &\leq (2+\lceil \log_2(n+1)\rceil)(12+n),\\
        \Size(\Phi_\varepsilon) &\leq pc\left(\frac{\varepsilon}{\sqrt T}\right)^{-\frac{1}{n+1}}.
    \end{align*}
\end{lemma}
\begin{proof}
    By the Cauchy-Schwarz inequality one has
    \begin{align*}
        \left|\sum_{j=1}^p x_{k_j^*}(t)\widetilde{\Psi_j} - \sum_{j=1}^p \widetilde{x_j}(t) \widetilde{\Psi_j} \right|^2 &= \left|\sum_{j=1}^p (x_{k_j^*}(t) - \widetilde{x_j}(t))\widetilde{\Psi_j} \right|^2\\
        &\leq \left(\sum_{j=1}^p |x_{k_j^*}(t) - \widetilde{x_j}(t)|^2 \right) \left(\sum_{j=1}^p \widetilde{\Psi_j}^2\right).
    \end{align*}
    Taking the expectation and integrating with respect to $t$ leads to
    \begin{equation*}
        \int_0^T \E\left[\left|\sum_{j=1}^p x_{k_j^*}(t)\widetilde{\Psi_j} - \sum_{j=1}^p \widetilde{x_j}(t) \widetilde{\Psi_j} \right|^2 \right]\dd t \leq \left(\sum_{j=1}^p \|x_{k_j^*}-\widetilde{x_j}\|_{L^2([0,T])}^2\right)\left(\sum_{j=1}^p \E[\widetilde{\Psi_j}^2]\right).
    \end{equation*}

    Recall that $\mu,\sigma,e_j$ are $n$-times continuously differentiable with respect to their variables.
    By \cref{corollary:approx_1d_odes}, for any $\varepsilon \in (0, \sqrt{pT}/2)$ there exists a neural network $\Phi_\varepsilon \in \mathcal N_{1,p}$ such that
    \begin{align*}
        \Depth(\Phi_\varepsilon) \leq (2+\lceil \log_2(n+1)\rceil)(12+n),\quad
        \Size(\Phi_\varepsilon) \leq pc\left(\frac{\varepsilon}{\sqrt T}\right)^{-\frac{1}{n+1}}.
    \end{align*}
    Moreover,
    \begin{equation*}
        \sum_{j=1}^p \|x_{k_j^*}-\widetilde{x_j}\|_{L^2([0,T])}^2 \leq \varepsilon^2. 
    \end{equation*}
\end{proof}

\subsection{D-dimensional SDE}\label{sec:d dimensional}

Let $H = L^2([0,T]; \R^d) \cong L^2([0,T]) \otimes \R^d$ and $B := (B_t^1, \dots, B_t^d)_{t \in [0,T]}$ be a $d$-dimensional Brownian motion defined on $(\Omega, \mathcal F, (\mathcal F_t)_{t \in [0,T]}, \mathbb P)$, where $(\mathcal F_t)_t$ is its natural filtration and $\mathcal F := \sigma(W(h) : h \in H)$. Then,
\begin{equation*}
    W(h) := \sum_{k=1}^d \int_0^T h^k(t) \dd W_t^k
\end{equation*}
is an \emph{isonormal Gaussian process} for $H$.

Let $(\phi_{ij})_{j=1,i \geq 0}^{j=d}$ be an orthonormal basis of $L^2([0,T]; \R^d)$.
For instance, $\phi_{ij} = \psi_i \otimes e_j$, where $(\psi_i)_{i \geq 0}$ is an orthonormal basis of $L^2([0,T])$ and $(e_j)_{j=1}^d$ is the canonical basis of $\R^d$.
Then, any random variable $F \in L^2(\Omega, \mathcal F, \mathbb P)$ admits the following Wiener chaos expansion
\begin{equation*}
    F = \sum_{k=0}^\infty \sum_{\substack{|\alpha|=k\\\alpha \in \mathcal I^d}} f_\alpha \left(\underbrace{\prod_{j=1}^d \prod_{i=1}^\infty H_{\alpha_i^j}(W(\phi_{ij}))}_{=:\Psi_\alpha}\right).
\end{equation*}

Consider the following $d$-dimensional It\^o process
\begin{equation*}
    \dd\mathbf X_t = \mathbf\mu(t, \mathbf X_t)\dd t + \mathbf\sigma(t, \mathbf X_t)\cdot\dd\mathbf B_t,
\end{equation*}
which in integral form reads
\begin{equation*}
    \mathbf X_t = \mathbf X_0 + \int_0^t \mu(s, \mathbf X_s)\dd s + \int_0^t \sigma(s, \mathbf X_s)\cdot \dd\mathbf B_s.
\end{equation*}
The Wiener chaos expansion for each component $X_t^j$ of $\mathbf X_t$ is given by
\begin{equation*}
    X_t^j = \sum_{k=0}^\infty \sum_{\substack{|\alpha|=k\\\alpha \in \mathcal I^d}} x_\alpha^j(t) \left(\prod_{j=1}^d \prod_{i=1}^\infty H_{\alpha_i^j}(W(\phi_{ij}))\right),\qquad 1\leq j\leq d.
\end{equation*}
Following the proof of \cref{theorem:propagator_system}, it is possible to show that the $x_\alpha^j(t)$ satisfy a system of ordinary differential equation.

Suppose that we want to approximate each component $X_t^j$ with $p$ coefficients.
Then, the analysis for the approximation of the coefficients $x_\alpha^j(t)$ \cref{lemma:nn_reconstruction} does not change qualitatively since by \cref{proposition:parallelisation} the size is only multiplied by $d$.
Since we have $d$ times more Hermite polynomials in the polynomial chaos $\Psi_\alpha$, by \cref{corollary:relu_approx_chaos} the size of the network again is multiplied by $d$.
    \section{Numerical experiments}
\label{sec:experiments}

This section is concerned with illustrating the practical properties of our SDEONet architecture.
In particular, it is shown to approximate the stochastic process $(X_t)_t$ at any time $0\leq t\leq T$ in numerical computations with a reasonable number of parameters.
To enforce the learning of the initial condition $X_0$, the training loss is extended by a second term, which then reads
\begin{equation}
    \mathcal L(\theta) = \frac 1 B \left(\sum_{i=1}^B \|X_{t_i}-R_\sigma \mathcal N_\theta^{m,p}(W^i, t_i)\|_2^2 + \sum_{i=1}^B \|X_0 - R_\sigma \mathcal N_\theta^{m,p}(W^i, 0)\|_2^2\right).
    \label{eq:loss}
\end{equation}
To assess the performance of our model, we use different metrics.
These are computed at each time step on the time grid to check if our model is able to approximate the stochastic process at any time $t$.
The first two are the absolute $L^2$ error $\|X_t - \tilde X_t\|_{L^2}$ and the relative $L^2$ error $\frac{\|X_t - \tilde X_t\|_{L^2}}{\|X_t\|_{L^2}}$.
They are approximated by a Monte-Carlo estimation
\begin{equation*}
    \|F\|_{L^2} \approx \left(\frac 1 N \sum_{i=1}^N |F^i|^2\right)^{1/2},
\end{equation*}
where $(F_i)_{i=1}^N$ are realisations of $F$.

Another reasonable metric to consider is the Wasserstein $2$-distance defined by
\begin{equation*}
    \mathcal W_2(\mu, \nu)^2 := \min_{\gamma \in \Pi(\mu, \nu)}\int_{\R^d}\|x-y\|_2^2 \dd\gamma(x,y),
\end{equation*}
where $\Pi(\mu,\nu) := \left\{\gamma \in \mathcal P(\R^d \times \R^d) : (\pi_0)_\sharp \gamma = \mu, (\pi_1)_\sharp \gamma = \nu \right\}$ is the \textit{transport plan} and $\pi_0$ and $\pi_1$ are the two projections of $\R^d \times \R^d$ onto its factors.
When $d=1$, it is possible to approximate it by considering the empirical measures $\mu_n = \frac 1 n \sum_{i=1}^n \delta_{X_i}$, $\nu_n = \frac 1 n \sum_{i=1}^n \delta_{Y_i}$ and then to compute
\begin{equation*}
    \mathcal W_2(\mu_n, \nu_n)^2 = \frac 1 n \sum_{i=1}^n \|X_{(i)} - Y_{(i)}\|_2^2,
\end{equation*}
where $X_{(1)} \leq X_{(2)} \leq \dots \leq X_{(n)}$.
In higher dimensions $d > 1$, the computation of $W_2$ is quite elaborate.
Approximations can be obtained by means of the \emph{Sinkhorn algorithm}~\cite{NIPS2013_af21d0c9,chizat2020faster}.
This algorithm makes use of \emph{Sinkhorn divergence}, as described in~\cite{chizat2020faster}, and allows the computation of an approximation of the Wasserstein distance without the penalty of the dimension.
\subsection{1D processes}
\label{sec:1D}

In the next experiments the model is defined by $m=32$, $p=64$, 2 hidden layers of 256 neurons each.
The model is learned on a dataset of 20,000 samples of $X_t$ with $t \sim U(0, T)$ during 30 epochs with a learning rate of $3\cdot 10^{-4}$ with the Adam optimizer~\cite{Kingma2014AdamAM} and a batch size of $64$.

\subsubsection{Ornstein-Uhlenbeck process}
\label{sec:OU}

The Ornstein-Uhlenbeck process is a crucial stochastic process in the area of mathematical physics and stochastic calculus.
It is a continuous-time stochastic process that finds extensive application in emulating a diverse range of phenomena across multiple fields such as physics, finance, biology, and engineering.
This process is useful for modelling mean-reversing behaviour, where a variable tends to return to its mean over time.
This makes it a valuable tool for understanding and modelling stable and self-correcting processes. 
It is defined by
\begin{equation*}
    \dd X_t = -\theta(X_t - \mu)\dd t + \sigma \dd W_t,
\end{equation*}
where $\theta >0$, $\mu$ and $\sigma > 0$ are parameters.
Using Itô's formula, it is possible to get an explicit expression of $X_t$ given $X_0$ by
\begin{equation*}
    X_t = X_0 e^{-\theta t} + \mu(1 - e^{-\theta t}) + \frac{\sigma}{\sqrt{2\theta}} W_{1 - e^{-2\theta t}}.
\end{equation*}
For the numerical experiments, we chose $\theta=1$, $\mu=1.2$, and $\sigma = 1.3$.

\begin{figure}[t]
    \centering
    \subfloat[Geometric Brownian motion\label{fig:gbm_trajectories}]{
        \includegraphics[width=0.45\linewidth]{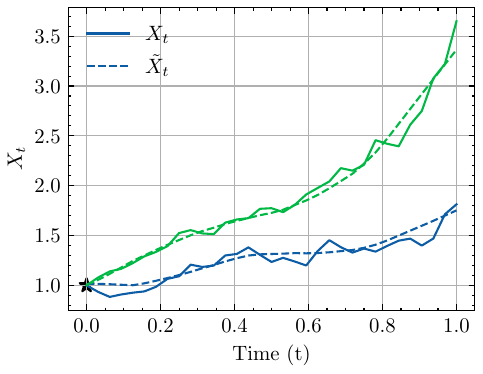}
    }
    \subfloat[Ornstein-Uhlenbeck process\label{fig:ou_trajectories}]{
        \includegraphics[width=0.45\linewidth]{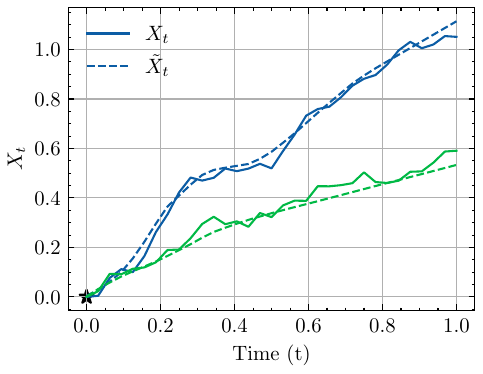}
    }
    \caption{Trajectories of the true processes $X_t$ and the approximations $\tilde X_t$.}
    \label{fig:trajectories}
\end{figure}

Examining the trajectories depicted in \cref{fig:ou_trajectories}, we observe that the approximation $\tilde X_t$ exhibits a notable smoother behaviour compared to the true stochastic process $X_t$.
This behaviour is likely due to a small dimension of polynomial chaos.
Moreover, we can discern the impact of the additional term incorporated into the loss function.
This supplementary term plays a crucial role in aiding the model to effectively learning the initial state, represented by $X_0$, which is confirmed by \cref{fig:ou_l2_loss}.
\cref{fig:ou_w2} also shows that the model is able to learn the stochastic process at each time step $t$ very accurately since the random variables $X_t$ and $\tilde X_t$ are close in distribution.

\begin{figure}[ht]
    \centering
    \subfloat[$L^2$ loss (blue) and relative $L^2$ loss (red) over time\label{fig:ou_l2_loss}]{
        \includegraphics[width=0.5\linewidth]{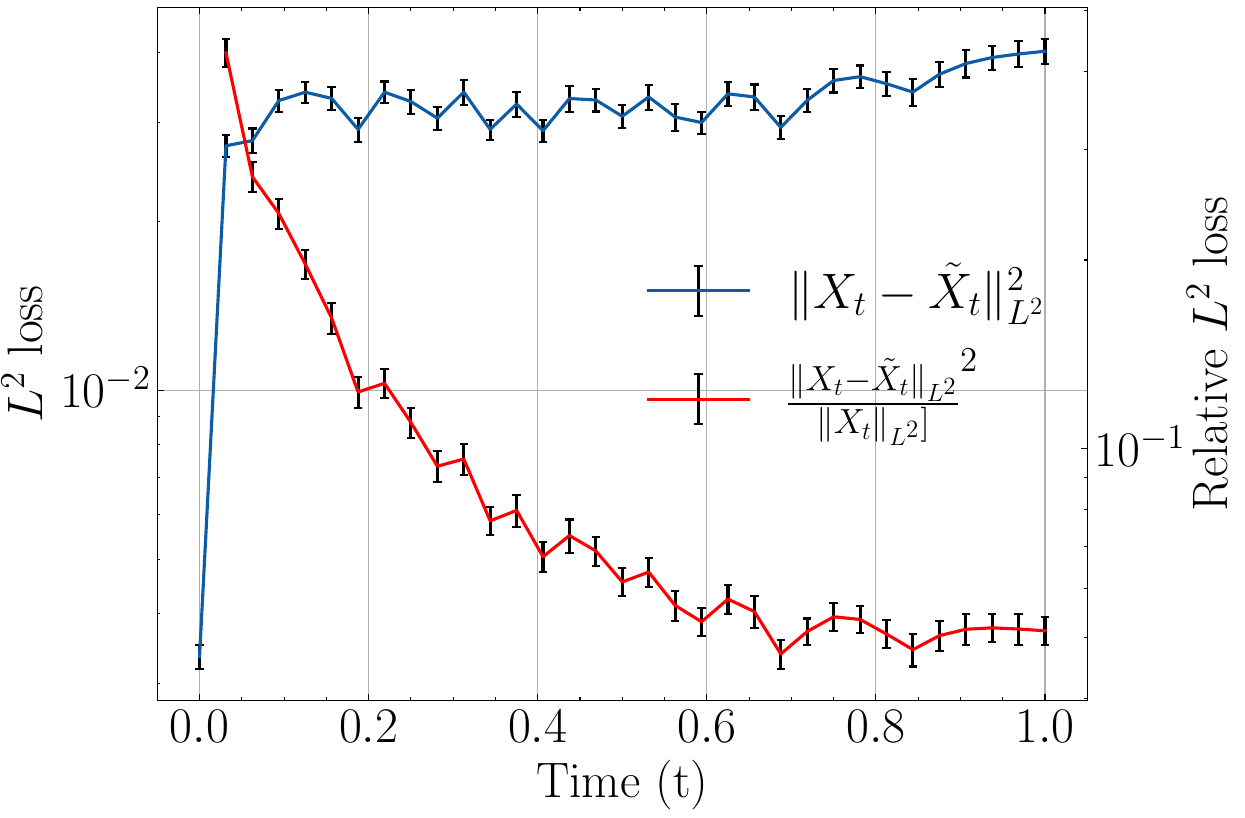}
    }
    \subfloat[$W_2(X_t, \tilde X_t)$ over time\label{fig:ou_w2}]{
        \includegraphics[width=0.43\linewidth]{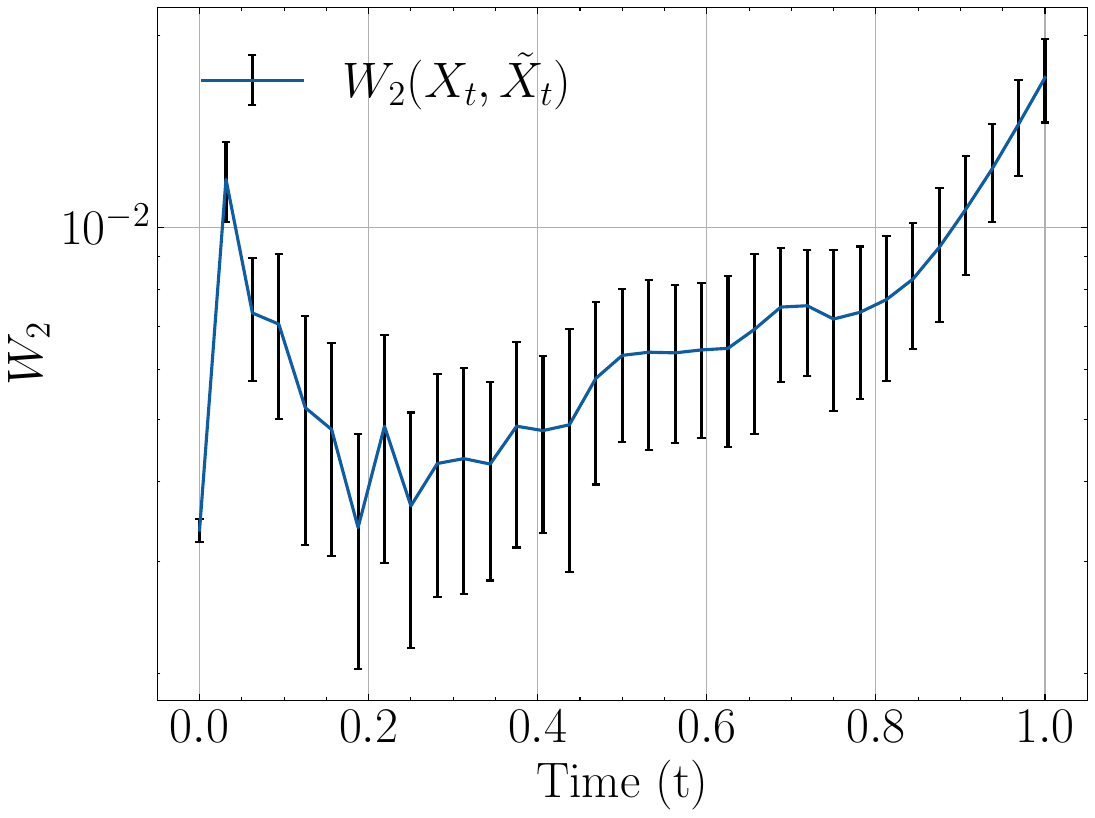}
    }
    \caption{$L^2$ loss and Wasserstein 2-distance over time for the Ornstein-Uhlenbeck process, computed over 2,000 samples and averaged over 100 independent realisations. The error bars correspond to $3\sigma$.}
    \label{fig:ou_loss}
\end{figure}

\subsubsection{Geometric Brownian motion}
\label{sec:geometric brownian}

The Geometric Brownian motion (GBM) is a widely used stochastic process in finance, mathematical modelling, and statistical physics.
In particular, this process is crucial for understanding and describing the pricing of financial assets and to model stock prices.
GBM represents an extension of the classic Brownian movement, which incorporates exponential growth and is characterised by its capacity to capture the innate uncertainty and volatility linked with genuine financial markets.

The Geometric Brownian motion is defined by
\begin{equation*}
    \dd X_t = \mu X_t \dd t + \sigma X_t \dd W_t,
\end{equation*}
where $\mu$ and $\sigma > 0$ are parameters.
Using Itô's formula, it is possible to obtain an explicit expression of $X_t$ given $X_0$, namely
\begin{equation*}
    X_t = X_0 \exp\left(\left(\mu - \frac{\sigma^2}{2}\right)t + \sigma W_t \right).
\end{equation*}

For the numerical experiments, we chose $\mu=1.0$ and $\sigma = 0.3$.

The plot in \cref{fig:gbm_trajectories} shows a similar behaviour as for the Ornstein-Uhlenbeck process above.
However, we see that the approximation is slightly worse due to the nature of the Geometric Brownian motion.
In practice, we also notice that for large $\sigma$ it becomes more difficult to learn the stochastic process.

\cref{fig:gbm_l2_loss} shows that with this choice of parameters $\mu$ and $\sigma$ the model is able to learn the stochastic process with a small $L^2$ and relative $L^2$ error.
Moreover, \cref{fig:gbm_w2} illustrates that the random variables $X_t$ and $\tilde X_t$ are close in distribution at each time step $t$.
This can be considered a more appropriate metric to assess the accuracy of the learned operator model.

\begin{figure}[ht]
    \centering
    \subfloat[$L^2$ loss (blue) and relative $L^2$ loss (red) over time for the approximation of the GBM\label{fig:gbm_l2_loss}]{
        \includegraphics[width=0.5\linewidth]{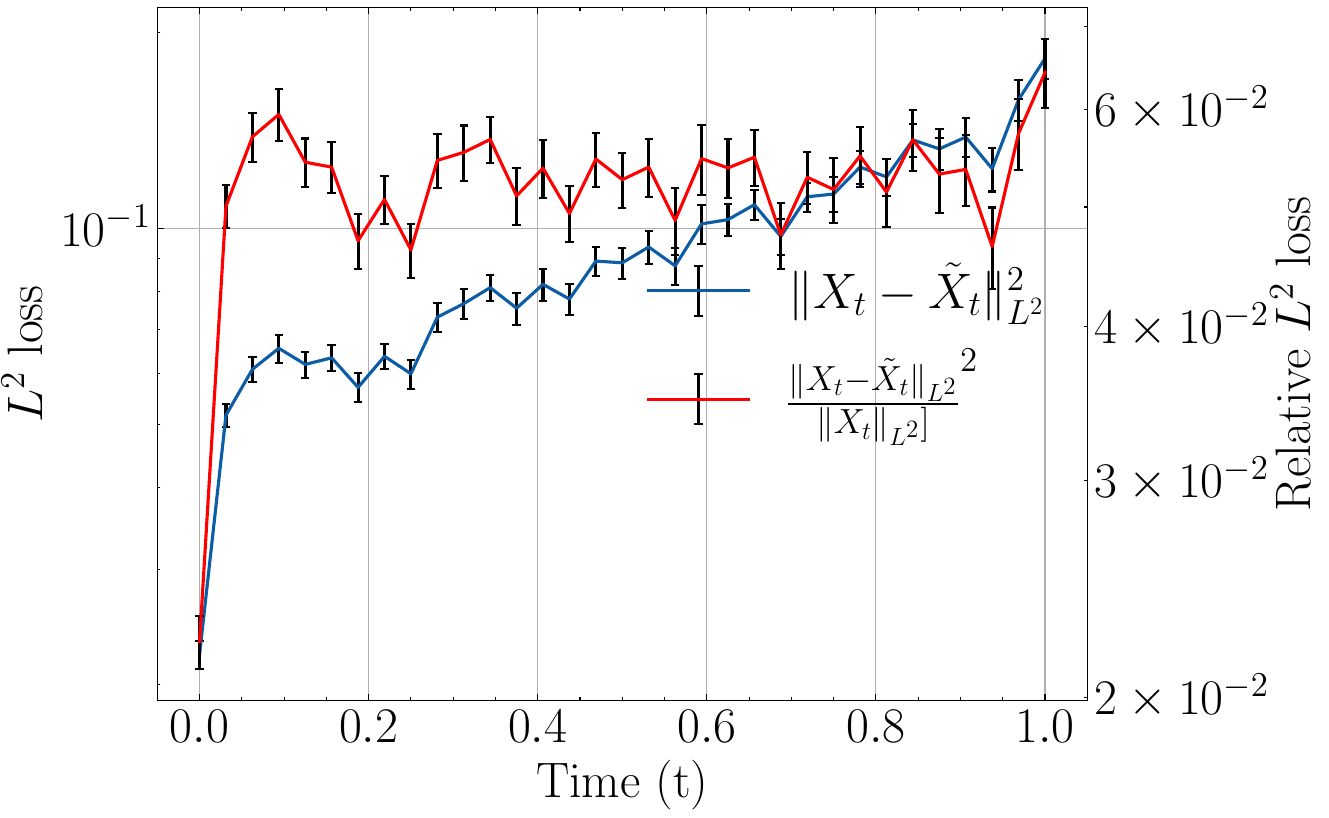}
    }
    \subfloat[$W_2(X_t, \tilde X_t)$ over time\vspace{3ex}\label{fig:gbm_w2}]{
        \includegraphics[width=0.4\linewidth]{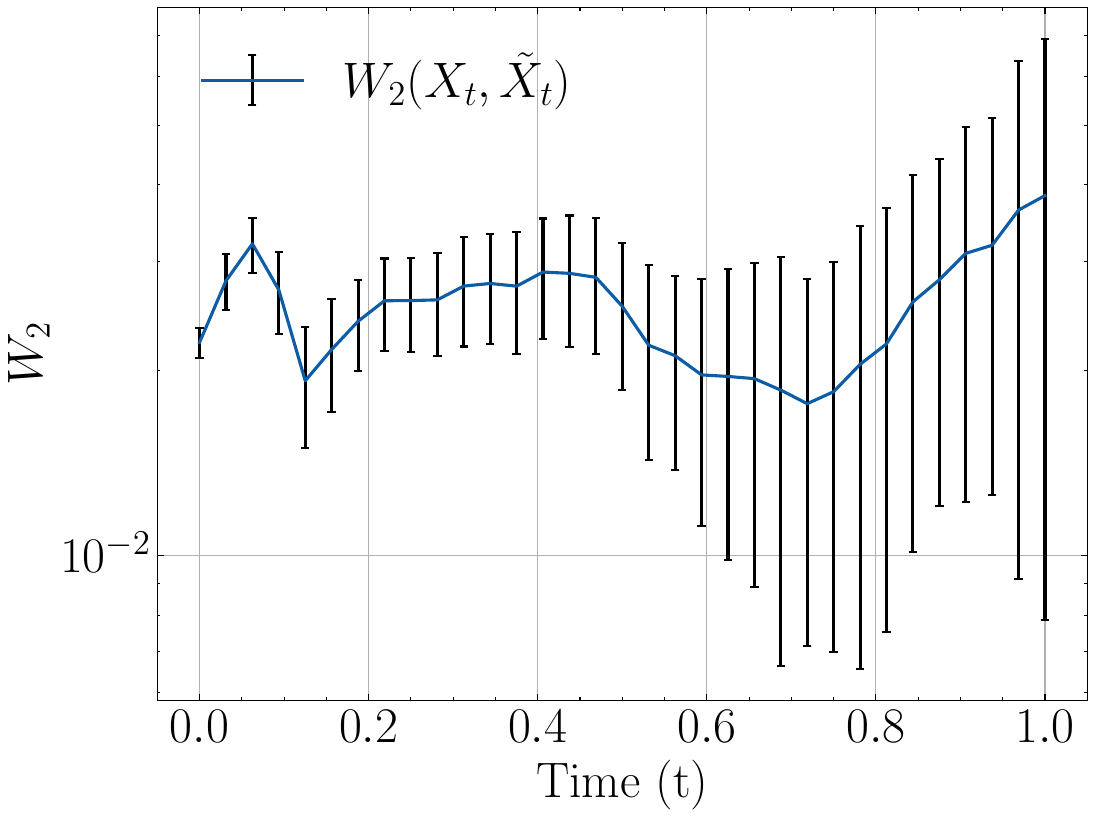}
    }
    \caption{$L^2$ loss and Wasserstein 2-distance over time for the Geometric Brownian motion computed over 2,000 samples and averaged over 100 independent realisations. The error bars correspond to $3\sigma$.}
    \label{fig:gbm_loss}
\end{figure}

\subsection{Multi-dimensional}
\label{sec:multi_d}
We consider the Langevin process for a multidimensional experiment, using the potential of the multivariate Normal distribution. This process describes a particle subject to a potential force, friction, and random noise.
Let $X_t$ be the particle's position at time $t$, with potential $V \in C^1(\R^d)$. The Langevin equation is given by,
\begin{equation*}
    m \frac{\dd^2 X_t}{\dd t^2} = -\nabla V(X_t) - \lambda m \frac{\dd X_t}{\dd t} + \sqrt{kT} \frac{\dd B_t}{\dd t},
\end{equation*}
where $m$ is the mass, $\lambda$ is the friction coefficient, $k$ is the Boltzmann constant, $T$ is the temperature, and $B_t$ is the standard Brownian motion.
Assuming high friction and slow particle movement, we can simplify to,
\begin{equation*}
    \dd X_t = -\nabla V(X_t)\dd t + \sqrt{2}\dd B_t,
\end{equation*}
For an ensemble of particles, we investigate the convergence of their distribution as $t \to \infty$. Assuming $V$ is $m$-strongly convex and the initial distribution has finite second moment, a stationary distribution exists \cite{bj/1178291835}.
The Fokker-Planck equation for the probability density $p(\cdot, t)$ of $X_t$ is,
\begin{equation*}
    \frac{\partial p}{\partial t}(x, t) = \operatorname{div}(\nabla V(x) p(x,t) + \nabla_x p(x,t)),
\end{equation*}
where $\operatorname{div}(f) := \sum_{j=1}^d \partial_j f$. From this equation, we immediately get the stationary distribution,
\begin{equation*}
    p_\infty := \frac{\exp(-V)}{\int \exp(-V) }.
\end{equation*}
It can also be shown that the convergence is exponentially fast.

Consider the multivariate normal distribution with mean $\mu$, covariance matrix $\Sigma$ and probability density function
\begin{equation*}
    \gamma(x) := (2\pi)^{-d/2} \det(\Sigma)^{-1/2} \exp\left(-\frac 1 2 (x-\mu)^T \Sigma^{-1} (x-\mu)\right).
\end{equation*}
The associated potential $V := -\log\gamma$ is given by
\begin{equation*}
    V(x) = \frac d 2 \ln(2\pi) + \frac 1 2 \ln(\det \Sigma) + \frac 1 2 (x-\mu)^T \Sigma^{-1}(x-\mu).
\end{equation*}
Therefore, the respective Langevin process can be written as
\begin{equation*}
    \dd X_t = -\Sigma^{-1}(X_t - \mu)\dd t + \sqrt{2}\dd B_t.
\end{equation*}

For the numerical experiment, we have chosen $d = 5$, $\Sigma = I$ and $\mu = 2 ((-1)^j j)_{j=0}^{d-1}$.
The neural network has the same architecture as in the one dimensional case, except that we now consider $64$ coefficients for each component of the approximation $\tilde X_t$.

\begin{figure}[ht]
    \centering
    \subfloat[$L^2$ loss (blue) and relative $L^2$ loss (red)\label{fig:langevin_l2_loss}]{
        \includegraphics[width=0.5\linewidth]{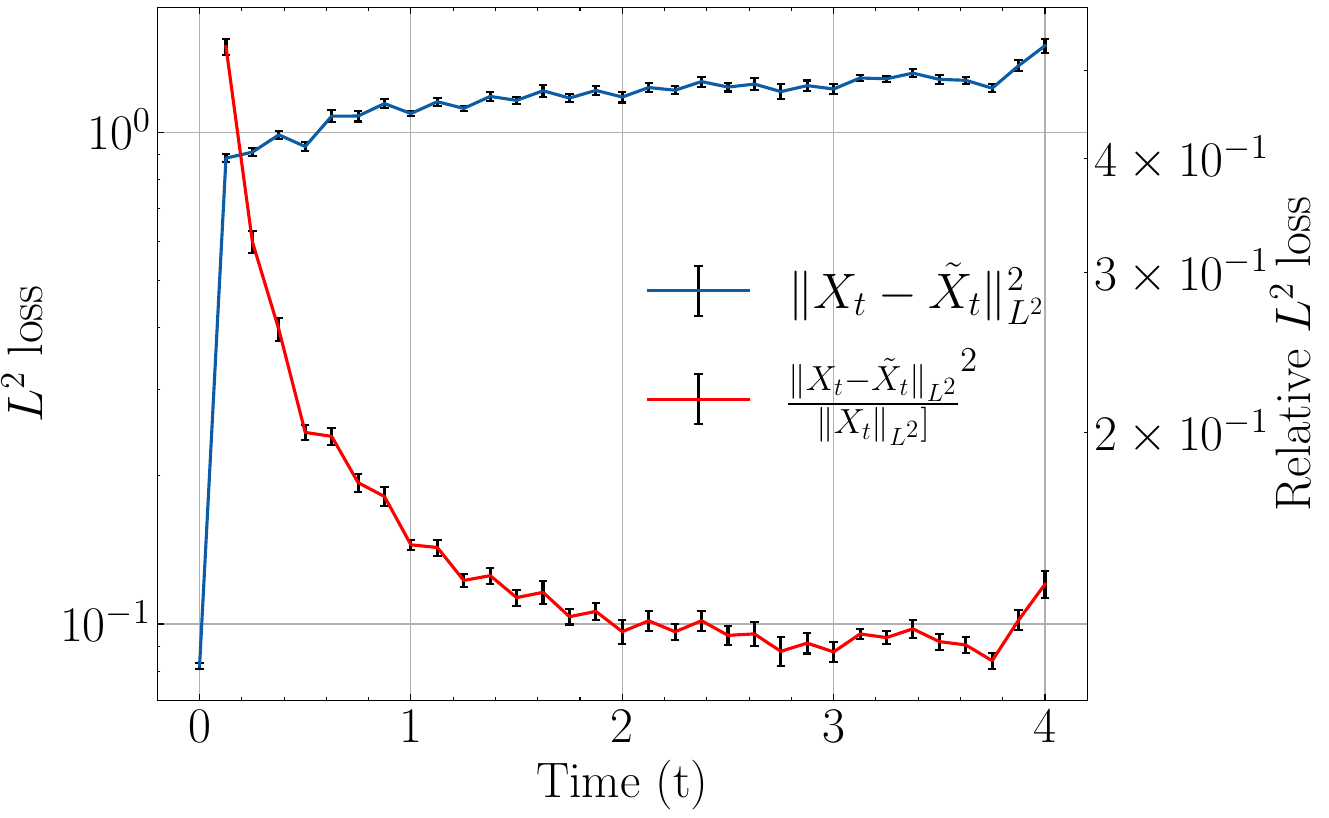}
    }
    \subfloat[$W_2(X_t, \tilde X_t)$ distance\label{fig:langevin_w2}]{
        \includegraphics[width=0.4\linewidth]{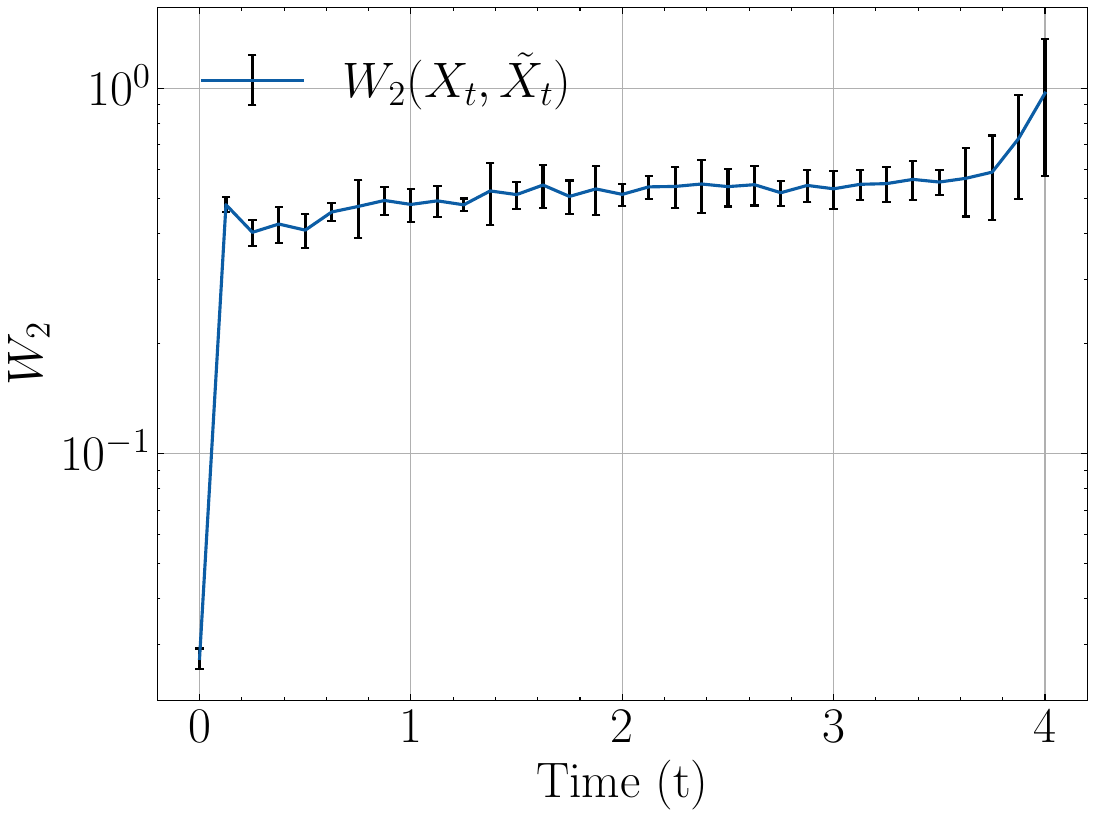}
    }
    \caption{$L^2$ loss and Wasserstein 2-distance over time for the Langevin process, computed over 2,000 samples and averaged over 10 independent realisations. The error bars correspond to $3\sigma$.}
    \label{fig:langevin_loss}
\end{figure}

The error plots in \cref{fig:langevin_loss} show that the model is still able to approximate the stochastic process $X_t$ in higher dimensions, even though the task is structurally more complicated.
We expect that an extended training effort would lead to an even better operator model.
    \section{Conclusion}

In this work, we have developed a new NN architecture called SDEONet to approximate the solution of a SDE \cref{eq:sde} using the notion of DeepONets \cite{luDeepONetLearningNonlinear2021} and a polynomial chaos expansion \cref{eq:chaos_expansion}, which is a standard technique for representing stochastic fields in UQ. Classical methods using polynomial chaos expansion for solving SDE \cite{huschtoAsymptoticErrorChaos2019} struggle to handle the $\binom{p+m}{p}$ coefficients, which grows very quickly when the number of basis elements $m$ or the maximal degree $p$ is increased.
It hence is inevitable to devise an appropriate truncation and compression, that still allows for accurate results in practice.
The method we have developed is a new strategy to learn a sparse Wiener chaos expansion of the solution of the SDE. The analysis shows that the size of the required neural networks is quite small due to the regularity of the coefficients and the use of Hermite polynomials.

The numerical experiments show promising results with small relative $L^2$ and $\mathcal W_2$ errors. However, the stability of the model should be improved when the process has a large variance, like with Geometric Brownian motion. Concerning the multidimensional case, as discussed in \Cref{sec:d dimensional} the experiment in \Cref{sec:multi_d} also suggests that our model is able to accurately approximate the solution of a multidimensional SDE without suffering from the ``curse of dimensionality''.

    \section*{Acknowledgements}
ME acknowledge partial funding from the Deutsche Forschungsgemeinschaft (DFG, German Research Foundation) in the priority programme SPP 2298 ``Theoretical Foundations of Deep Learning''. ME \& CM acknowledge funding by the ANR-DFG project ``COFNET: Compositional functions networks - adaptive learning for high-dimensional approximation and uncertainty quantification''.

    \bibliographystyle{amsplain}
    \bibliography{main}

    \appendix
    \section{Proofs of main results}

\setcounter{subsection}{1}

\begin{proof}[Proof of \cref{lemma:ub_trunc_error}]
    Let $m,n \in \N$ and $k = \lfloor \log_2(n) \rfloor$.
    Applying \cref{theorem:l2_error_truncation} leads to
    \begin{equation*}
        \E[(X_t - X_t^{m,n})^2] \leq C(t,K)(1+x_0^2)\left(\frac{1}{(m+1)!} + \sum_{\ell=k+1}^\infty \sum_{j=1}^{2^k} \left(E_{j,\ell}(t)^2 + \int_0^t E_{j,\ell}^2(\tau)\dd\tau\right) \right).
    \end{equation*}
    Integrating from $t=0$ to $T$ yields
    \begin{align*}
        \int_0^T \E[|X_t - X_t^{m,n}|^2]\dd t 
        \leq &\int_0^T C(t,K)(1+x_0^2)\\
        &\times \left(\frac{1}{(m+1)!} + \sum_{\ell=k+1}^\infty\sum_{j=1}^{2^k} \left(E_{j,\ell}(t)^2 + \int_0^t E_{j,\ell}^2(\tau)\dd\tau\right)\right)\dd t.
    \end{align*}
    Then, using \cref{proposition:decay_rest}, we have
    \begin{align*}
        &\int_0^T C(t,K)(1+x_0^2) \left(\frac{1}{(m+1)!} + \sum_{\ell=k+1}^\infty\sum_{j=1}^{2^k} \left(E_{j,\ell}(t)^2 + \int_0^t E_{j,\ell}^2(\tau)\dd\tau\right)\right)\dd t \\
        \leq& (1+x_0^2)\int_0^T C(t, K)\left(\frac{1}{(m+1)!} + \frac{2T(1+t)}{n}\right) \dd t.
    \end{align*}
    The results follows by definition of $k^*$.
\end{proof}

\begin{proof}[Proof of \cref{lemma:decay_squared_coef}]
    Note that
    \begin{equation*}
        X_t = \sum_{n=0}^\infty \sum_{|\alpha|=n}x_\alpha(t)\Psi_\alpha = \sum_{n=0}^\infty I_n(\xi_n(\mathbf t^n; t))
    \end{equation*}
    with $\mathbf t^n := (t_1,\dots,t_n)$. Since $X_t$ is infinitely Malliavin differentiable, by \cref{theorem:kernel_stroock} the symmetric kernel functions are given by $\xi_n(\cdot; t) : \mathbf t^n \mapsto \frac{1}{n!}\E[D_{t_1,\dots,t_n}^n X_t]$ with $n$-th order \emph{Malliavin derivative} of $X_t$ denoted by $D^n X_t$. $I_n$ is the multiple stochastic integral of order $n$ introduced in \cref{eq:multiple_stochastic_integrals}.
    It follows that
    \begin{align*}
        \sum_{\ell=0}^m\sum_{|\alpha|=\ell}x_\alpha(t)^2&= \sum_{\ell=0}^m \E[(I_\ell(\xi_\ell(\mathbf t^\ell; t)))^2]\\
        &=\sum_{\ell=0}^m \ell! \langle \xi_\ell(\mathbf t^\ell; t), \xi_\ell(\mathbf t^\ell; t) \rangle_{L^2([0,T]^\ell)}\\
        &= \sum_{\ell=0}^m \int^{(\ell);t}\E[(D_{t_1,\dots,t_\ell}^\ell X_t)^2]\dd \mathbf t^\ell\\
        &\leq (1+x_0^2)\sum_{\ell=0}^m C^\ell e^{C\ell t} \int^{(\ell);t} \dd \mathbf t^\ell\\
        &= (1+x_0^2)\sum_{\ell=0}^m \frac{(Cte^{Ct})^\ell}{\ell!},
    \end{align*}
    with $C = C(K, T)$ and $\int^{(\ell);t} f(\cdot)\dd\mathbf t^\ell := \int_0^t \int_0^{t_\ell} \dots \int_0^{t_2} f(\cdot) \dd t_1 \dots \dd t_\ell$.
    The inequality is derived by using \cref{proposition:malliavin_bound,theorem:square_int_process}.
    Now, note that $\sum_{\ell=0}^m \frac{(Cte^{Ct})^\ell}{\ell!} \leq e^{Cte^{Ct}} - \frac{(Cte^{Ct})^{m+1}}{(m+1)!}$ using the Taylor-Lagrange formula.
    With this, we obtain
    \begin{equation*}                   \sum_{\ell=0}^m\sum_{|\alpha|=\ell}x_\alpha(t)^2 \leq \left((1+x_0^2)(e^{Cte^{Ct}} - \frac{(Cte^{Ct})^{m+1}}{(m+1)!}\right).
    \end{equation*}
\end{proof}

\begin{proof}[Proof of \cref{lemma:ub_approx_error}]
    By the Cauchy-Schwarz inequality, it follows that
    \begin{align*}
        \left|\sum_{j=1}^p x_{k_j^*}(t)\Psi_{k_j^*} - \sum_{j=1}^p x_{k_j^*}(t)\widetilde{\Psi_j}\right|^2 &= \left|\sum_{j=1}^p x_{k_j^*}(t)(\Psi_{k_j^*}-\widetilde{\Psi_j})\right|^2\\
        &\leq \left(\sum_{j=1}^p x_{k_j^*}(t)^2\right)\left(\sum_{j=1}^p |\Psi_{k_j^*}-\widetilde{\Psi_j}|^2\right).
    \end{align*}
    Taking the expectation and integrating with respect to $t$ leads to
    \begin{align*}
        \int_0^T \E\left[\left|\sum_{j=1}^p x_{k_j^*}(t)\Psi_{k_j^*} - \sum_{j=1}^p x_{k_j^*}(t)\widetilde{\Psi_j}\right|^2\right]\dd t &\leq \int_0^T \left(\sum_{j=1}^p x_{k_j^*}(t)^2\right)\E\left[\sum_{j=1}^p |\Psi_{k_j^*}-\widetilde{\Psi_j}|^2\right]\dd t.
    \end{align*}
    The next step is to find an upper bound of $\int_0^T \left(\sum_{j=1}^p x_{k_j^*}(t)^2\right) \dd t$.
    By \cref{lemma:decay_squared_coef}, with $(m,n) \in J_p$ we have
    \begin{equation*}
        \sum_{\ell = 0}^m \sum_{|\alpha|=\ell} x_\alpha(t)^2 \leq (1+x_0^2)\left(e^{Ct e^{Ct}} - \frac{(Cte^{Ct})^{m+1}}{(m+1)!}\right).
    \end{equation*}
    Integrating with respect to $t$ results in
    \begin{align*}
        \int_0^T \sum_{\ell=0}^m\sum_{|\alpha|=\ell}x_\alpha(t)^2 \dd t &\leq (1+x_0^2)\left(\underbrace{\int_0^T e^{Cte^{Ct}} \dd t}_{=: A(K,T)} - \int_0^T \frac{(Cte^{Ct})^{m+1}}{(m+1)!}\dd t\right)\\
        &= (1+x_0^2)\left(A(K,T) - \frac{1}{(m+1)!}\sum_{\ell=0}^{m+1}\binom{m+1}{\ell}\int_0^T (Ct)^{2(m+1)-\ell}\dd t\right)\\
        &= (1+x_0^2)\left(A(K,T) - \frac{1}{(m+1)!}\sum_{\ell=0}^{m+1}\binom{m+1}{\ell}\frac{T(CT)^{2(m+1)-\ell}}{2(m+1)-k+1}\right)\\
        &\leq (1+x_0^2)\left(A(K,T) - \frac{1}{(m+1)!}\sum_{\ell=0}^{m+1}\binom{m+1}{\ell}\frac{T(CT)^{2(m+1)-\ell}}{2(m+1)+1}\right)\\
        &= (1+x_0^2)\left(A(K,T) - \frac{1}{(m+1)!}\frac{T(CT)^{2(m+1)}}{2(m+1)+1}\left(1+\frac{1}{CT}\right)^{m+1}\right).
    \end{align*}
    By definition of $k^*$ it thus follows for all $(m,n) \in J_p$ that
    \begin{equation*}
        \int_0^T \left(\sum_{j=1}^p x_{k_j^*}(t)^2\right) \dd t \leq (1+x_0^2)\left(A(K,T) - \frac{1}{(m+1)!}\frac{T(CT)^{2(m+1)}}{2(m+1)+1}\left(1+\frac{1}{CT}\right)^{m+1}\right).
    \end{equation*}

    Combining the above results, we obtain
    \begin{align*}
        \hat E_\text{approx}^2 &\leq \E\left[\sum_{j=1}^p |\Psi_{k_j^*}-\widetilde{\Psi_j}|^2\right]\\ &\qquad\times \min_{(m,n) \in J_p}(1+x_0^2)\left(A(K,T) - \frac{1}{(m+1)!}\frac{T(CT)^{2(m+1)}}{2(m+1)+1}\left(1+\frac{1}{CT}\right)^{m+1}\right).
    \end{align*}
\end{proof}

\begin{proof}[Proof of \cref{corollary:approx_1d_odes}]
    \textbf{Case $m=1$}.\\
    Under the assumptions on $f$, problem~\cref{eq:1d_ode} has a unique solution $X : [t_0, t_1] \to \R$ by the Cauchy-Lipschitz (Picard-Lindel\"of) theorem.
    By induction on $k$ it can be shown that $X \in C^{k+1}$.
    Note that the solution $X$ is $(k,1)$-H\"older continuous.
    Let 
    \begin{equation*}
        T : \left\{
        \begin{array}{rl}
            [-1/2,1/2] & \longrightarrow [t_0,t_1] \\
            x & \longmapsto (x+\frac 1 2)(t_1-t_0) + t_0,
        \end{array}
        \right.
    \end{equation*}
    which is a diffeomorphism. Define $Y = X \circ T : [-1/2,1/2] \to \R$, which is $(k,1)$-H\"older continuous on $[-1/2,1/2]$.
    Let $B := \max_{\alpha \in \{0,\dots,k\}} \left\|\frac{\dd^\alpha Y}{\dd x^\alpha}\right\|_\infty$.
    Applying \cref{theorem:approx_holder_nn}, there exists a constant $c = c(k, B) > 0$ and a neural network $\Phi_\varepsilon^Y$ such that
    \begin{align*}
        \Depth(\Phi_\varepsilon^Y) &\leq (2 + \lceil \log_2(k + 1)\rceil)(12+k),\\
        \Size(\Phi_\varepsilon^Y) &\leq c\left(\frac{\varepsilon}{(t_1-t_0)^{1/p}}\right)^{-\frac{1}{k+1}}.
    \end{align*}
    Moreover,
    \begin{align*}
        \|R_\sigma(\Phi_\varepsilon^Y) - Y\|_{L^p([-1/2,1/2])} &< \frac{\varepsilon}{(t_1-t_0)^{1/p}},\\
        \|R_\sigma(\Phi_\varepsilon^Y)\|_\infty &\leq \lceil B \rceil.
    \end{align*}
    Next, note that
    \begin{align*}
        \|R_\sigma(\Phi_\varepsilon^Y) - Y\|_{L^p([-1/2,1/2])}^p &= \frac{1}{t_1-t_0}\int_{t_0}^{t_1} (R_\sigma(\Phi_\varepsilon^Y)(T^{-1}(x)) - Y(T^{-1}(x)))^p \dd x\\
        &=\frac{1}{t_1-t_0}\|R_\sigma(\Phi_\varepsilon^Y)\circ T^{-1} - X\|_{L^p([t_0,t_1])}^p.
    \end{align*}
    If we write $\Phi_\varepsilon^Y = ((W^1, b^1),\dots,(W^L,b^L))$ then 
    \begin{equation*}
        \Phi_\varepsilon^X = ((\frac{1}{t_1-t_0}W^1, b^1 - \left(\frac{t_0}{t_1-t_0} + \frac 1 2\right)W^1), (W^2, b^2), \dots, (W^L, b^L))
    \end{equation*}
    satisfies $R_\sigma\Phi_\varepsilon^Y \circ T^{-1} = R_\sigma\Phi_\varepsilon^X$ and $\Size(\Phi_\varepsilon^Y) = \Size(\Phi_\varepsilon^X)$.
    Using the previous equation, we deduce that
    \begin{equation*}
        \|R_\sigma(\Phi_\varepsilon^X) - X\|_{L^p([t_0,t_1])} \leq \varepsilon.
    \end{equation*}

    \textbf{Case $m > 1$}.\\
    Let $(\Phi_\varepsilon^{X_k})_{k=1}^m$ be $m$ neural networks that approximate $X_k$ with accuracy $\frac{\varepsilon}{m^{1/p}}$ in the $L^p$ norm.
    Consider $\Phi_\varepsilon^X = P((\Phi_\varepsilon^{X_k})_{k=1}^m)$ as a parallelisation according to \cref{proposition:parallelisation} with $\Size(\Phi_\varepsilon^X) = \sum_{k=1}^m \Size(\Phi_\varepsilon^{X_k})$ and $R_\sigma \Phi_\varepsilon^X = (R_\sigma \Phi_\varepsilon^{X_k})_{k=1}^m$.
    Then,
    \begin{align*}
        \|R_\sigma \Phi_\varepsilon^X - X\|_{L^p([t_0,t_1])}^p &\leq \sum_{j=1}^m \|R_\sigma \Phi_\varepsilon^{X_k} - X_k\|_{L^p([t_0,t_1])}^p\\
        &\leq m \left(\frac{\varepsilon}{m^{1/p}}\right)^p\\
        &=\varepsilon^p.
    \end{align*}
\end{proof}
    
\end{document}